\documentclass[10pt,reqno]{amsart}

\usepackage[utf8]{inputenc}
        \usepackage[T1]{fontenc}
        \usepackage{lmodern}
        \usepackage{graphicx}
        \usepackage{caption}
         \usepackage{floatrow}

\usepackage{amsmath,amsthm,amscd,amssymb,bbm,graphicx, color}

\usepackage{mathrsfs} 
\usepackage{geometry}          
\usepackage{graphics,multicol}
\usepackage{epstopdf}
\usepackage{enumerate}

\usepackage{dsfont}
\newtheorem{remark}{Remark}

\newtheorem{theorem}{Theorem}
\newtheorem{definition}[theorem]{Definition}
\newtheorem{lemma}[theorem]{Lemma}
\newtheorem{proposition}[theorem]{Proposition}
\newtheorem{corollary}[theorem]{Corollary}

\def\T{\hat{T}}

\def\g {\gamma}

\usepackage{hyperref}

\newcommand{\sgn}{\mbox{\rm sgn}}							
\newcommand{\var}{\mbox{\rm var}}								

\newcommand{\ind}{\mathbbm{1}}

\usepackage{natbib}

\title{Number of visits in arbitrary sets for $\phi$-mixing dynamics}
\author{Sandro Gallo}
\address{Sandro Gallo: Departamento de Estat\'istica, Universidade Federal de S\~ao Carlos, CEP: 13565-905, S\~ao Paulo, Brasil}
\email{sandro.gallo@ufscar.br}
\author{Nicolai Haydn}
\address{Nicolai Haydn: Mathematics Department, USC, Los Angeles, 90089-2532.}
\email{nhaydn@math.usc.edu}
\author{Sandro Vaienti}
\address{Sandro Vaienti: Aix Marseille Université, Université de Toulon, CNRS, CPT, 13009 Marseille, France}
\email{vaienti@cpt.univ-mrs.fr}
\date{\today}
\begin{document}
\maketitle

\begin{abstract}
It is well-known that, for sufficiently mixing dynamical systems, the number of visits to balls and cylinders of vanishing measure is approximately Poisson compound distributed in the Kac scaling. Here we extend this kind of results when the target set is an arbitrary set with vanishing measure in the case of $\phi$-mixing systems. The error of approximation in total variation is derived using Stein-Chen method. An important part of the paper is dedicated to examples to illustrate the assumptions, as well as applications to temporal synchronisation of $g$-measures.  \end{abstract}

\tableofcontents


\section{Introduction}\label{sec:intro}
The  recurrence in small sets, which could be seen alternatively as a {\em rare} or {\em extreme} event, turned out to have very rich probabilistic features and established itself as a major statistical property of dynamical systems. We consider in this paper the general situation of a measurable deterministic dynamical system and try to characterise the distribution of the number of visits to sets whose measure will tend to zero. Since the probability to visit the set coincides with its measure for ergodic systems, one should normalise the length of the trajectory with the measure of the set, in order to get meaningful asymptotic distributions. We called it, in the paper, the Kac's scaling. If the system looses memory fast enough in the future, which is achieved with relatively strong mixing properties, the number of visits of a trajectory of length $n$ tends to follow a binomial distribution $B(n,p_n),$ where $p_n$ is the measure of the small set. Kac's scaling requires that the product $np_n$ equals asymptotically the constant $t$ and therefore one gets a Poisson law  of parameter $t$ in the limit of large $n$ for the number of visits up to time $\frac{t}{p_n}.$ The implementation of this heuristic  argument for a given measure preserving dynamical system, requires not only mixing properties, as we said above, but also some control on the {\em nature} of the small sets. When the map acts on a metric space, the small set is usually taken as  a ball around a given point $z$ and with radius shrinking to zero. The {\em nature} of the point $z$ could change the limit distribution. Suppose in fact that $z$ is a periodic point; even if the system is mixing, the orbits starting or passing close to $z$ tend to sojourn for a longer time in the small set. This produces an effect of {\em clusterization} which will alter the Poisson law  into a more general {\em compound Poisson distribution}.

The aim of the present paper is to obtain such results for measurable dynamical systems and for a wide class of small sets. The latter are obtained by fixing an initial measurable generating partition and by taking its  backward (and eventually forward for invertible systems), join. An arbitrary countable disjoint union of elements of the join of order $n$ will be a small set $U_n.$ We will also assume that the sequence $\{U_n\}_{n \ge1}$ is nested and that it converges to a set of measure zero. The asymptotic {distribution} of successive visits {to} $U_n$ will be assured by requiring that the invariant  measure is $\phi,$ or $\psi-$mixing with respect to the initial partition.

First of all we proceed to adapt  the Stein-Chen method \cite{chen2005stein, barbour1992compound, stein1986approximate, roos1994stein} to compare a given probability measure, in our case the distribution of the number $W_n$ of visits to a set,  to a compound Poisson distribution.  This will give us an error for the total variation distance between the two distribution. Any compound Poisson distribution depends upon a set of parameters $\lambda_l, l\ge 1.$ It has been shown in \cite{haydn2020limiting}, that those parameters are related to another sequence $\alpha_l, l\ge 1,$ (see Section \ref{sec:asympt}) which quantify the distribution of higher order returns. Whenever the limits defining the $\alpha_l$ exist and the latter verify a summable condition, the error term given by the Stein-Chen method will go to zero, and therefore we recover the expected convergence to a compound Poisson law: this is the content of the main result, Theorem \ref{main.theorem}. Applications to concrete examples  basically require to check two conditions on the system: (i) first of all the $\phi,$ or $\psi-$mixing property, which enters the estimate of the error in the Stein-Chen approach; (ii) secondly the existence and summability of the $\alpha_l$, which instead depend on the system and on the choice of the nested sequence of small sets $U_n$. A similar program was carried over in \cite{haydn2020limiting}, with a few substantial differences which in particular imply that the  examples given in the present paper cannot be covered by the theory developed in \cite{haydn2020limiting}. The latter targets $C^2$ local diffeomorphisms on smooth manifolds and satisfying a few geometrical and metric  conditions, among which the most relevant are: a)  local hyperbolicity  and distortion; b) the annulus-type condition which allows to control the relative measure of the neighborhoods of the small sets; and finally c) the decay of correlations which is stated in terms of Lipschitz against $L^{\infty}$ norms. The technique of the proof of \cite{haydn2020limiting} was different from the Stein-Chen' used here and it had a more geometric flavour, adapted to differentiable dynamical systems. In particular
it was possible to handle partially hyperbolic maps and synchronisation of coupled map lattices. In the latter case and for the invariant absolutely continuous
measure, it has been established that the returns to the diagonal is compound Poisson distributed
where the coefficients are given by certain integrals along the diagonal. This example is reconsidered in this paper and compared with a different way to collect points close to each other in the attempt to synchronise two or more trajectories. In the spirit of the present paper, a neighborhood of the diagonal will be constructed with the  elements of the join partition of increasing order, also called cylinders. As we said above, cylinders and union of cylinders will be our small sets. If the dynamical system is encoded in  a symbolic space, we could transport our theory in the domain of symbolic dynamics and cover new panels of examples which are unattainable with the previous geometric approach. Among the applications investigated in the paper, we quote here the {\em House of cards process}, for which the distribution of the number of visits to runs of length above a given threshold is found to be P\'olya-Aeppli, and a class of (not necessarily Markovian) regenerative processes for which we compute explicitly the parameters of the compound Poisson distribution. In particular  we exhibit the existence of the quantity $\alpha_1$ which takes on a particular role in {\em extreme value theory}, where it coincides with   the {\em extremal index}.

 An important part of the paper is dedicated to $g$-measures (see Section \ref{sec:examples2}). These are equilibrium states with normalized potentials of the form $\phi=\log g,$ where $g$ is the $g$-function \cite{keane/1972}. These objects form the counterpart, in the dynamical system setting, of the (possibly long memory)  stochastic processes. For this class of models, we give mild sufficient conditions (strict positivity and summable variation), allowing us to apply our theorem for the number of visits in cylinders around periodic points. It has been recently shown \cite{abadi/cardeno/gallo/2015} that for a particular class of $g$-measures called {\em renewal measures}, it is possible to show that the limit defining the extremal index does not exist even though the measure is $\phi$-mixing and this was due to   an essential discontinuity of $g$ in a given point. Here we will prove that uniform continuity is enough for the existence of the extremal index and, by discussing an example due to Furstenberg and Furstenberg, that the lack of continuity of $g$  does not prevent the existence of the parameter, leading to a P\'olya-Aeppli distribution around any periodic point. We will {then} consider a  decreasing cover by cylinders  of the diagonal in the $m$-dimensional product space where a given $g$-measure is seen as the coupling of the $m$ coordinates $g$-measures.  This will allow us to study the synchronisation of the coordinates, what we called {\em temporal synchronisation for $g$-measures}. In the general case where the coordinate $g$-measures are not independent, we will show the converge to a P\'olya-Aeppli distribution whose parameter is related to the topological pressure of a given potential, see Theorem \ref{theo:tempsync}. It is interesting to observe that whenever the coordinate $g$-measures are independent (the uncoupling case), the previous parameter can be expressed in terms of the Renyi entropy of order $m-1.$ We also address the more general question of the interaction of possibly distinct $g$-measures and we propose two ways to construct such an interaction.

Finally, in a discussion section on synchronisation, we highlight a difference between the geometric approach of \cite{haydn2020limiting} and the symbolic approach of the present work. In a simple example of two uncoupled identical deterministic dynamical systems, we show that the asymptotic distribution of synchronisations is P\'olya-Aeppli when we target the diagonal by cylinder sets (symbolic approach) but it is not P\'olya-Aeppli when we target the diagonal by tubular neighbourhoods (geometric approach). We conclude the discussion by considering yet another situation, of two uncoupled copies of the same Markov chain on $[0,1]$, with strong ergodicity conditions. We show that the asymptotic synchronisation always follows a pure Poisson distribution, meaning that there is no clustering phenomenon in that setting.\\

We previously compared our achievements with the results obtained in the  paper \cite{haydn2020limiting}; several other contributions deserve to be quoted and we will give here a brief survey of them. We should first remind  the seminal papers by   \cite{pitskel/1991} and   \cite{hirata/1993} who showed that generic points have, in the limit, Poisson distributed return times if one uses
cylinder neighbourhoods, while  at periodic points the return times distribution has a point
mass at the origin which corresponds to the periodicity of the point. This dichotomy inspired  and originated several successive works:  it  was proved in \cite{abadi/2001} for $\phi$-mixing systems in the symbolic setting, and in \cite{haydn2009distribution} for more general classes of dynamical
systems with various kind of mixing properties. The latter paper derived also the  error terms for the convergence to the limiting compound Poissonian distribution.  The extension to $\psi$-mixing shifts was given in \cite{kifer2014poisson}; for $\phi$-mixing systems a recent contribution is provided in \cite{kifer2018geometric}. The Chen-Stein method, which is at the base of the actual work, was firstly used in \cite{haydn2014return} for $\phi$-mixing measures and cylinder sets. A complementary approach to the statistics of the number of visits, has been developed in the framework of extreme value theory, where it is more often called point process, or particular kinds of it as the marked point process associated to extremal observations corresponding to exceedances of high
thresholds. See for instance \cite{freitas2013compound, freitas2018convergence, freitas2020point} for applications to deterministic and random dynamical systems, and the book \cite{lucarini2016extremes} for a panorama and an account
on extreme value theory and point processes applied to dynamical systems. The distribution of the number of visits to vanishing balls has been studied for systems modeled by a Young tower: in \cite{chazottes2013poisson} for the H\'enon attractor, in \cite{pene2016poisson}  for some nonuniformly hyperbolic
invertible dynamical systems, in \cite{Haydn2016Limiting} and \cite{haydn2017derivation} for polynomially decaying correlations.
Recurrence in billiards provided recently  several new contributions; for planar billiards in \cite{pene2010back} and \cite{freitas2014convergence}; in  \cite{pene2020spatio} the  spatio-temporal Poisson processes was obtained from recording not
only the successive times of visits to a set, but also the positions.

The paper is structured as follows. We directly state the main results in Section \ref{sec:results}, and follow with a discussion concerning assumptions and examples in Section \ref{sec:discussion}. Next, Section \ref{sec:examples2} specializes to the  case of $g$-measures. We provide one further discussion on synchronization in Section \ref{sec:geovscyl}, and conclude with Sections \ref{sec:proofs} and \ref{sec:remaining_results} containing  the proofs of all the results.

\section{Main results}\label{sec:results}

\noindent Let $T$ be a measurable map on a measure space $\Omega$ and 
$\mu$ a $T$-invariant measure on $\Omega$. Moreover let $\mathcal{A}$ be
a countable measurable partition on $\Omega$ and denote by $\mathcal{A}^n
 =\bigvee_{j=0}^{n-1}T^{-j}\mathcal{A}$ be the joins of $\mathcal{A}$.
 In the two-sided case when the map $T$ is invertible then the $n^{\text{th}}$ join is
 $\mathcal{A}^n=\bigvee_{j=-n+1}^{n-1}T^{-j}\mathcal{A}$.
We assume that  $\mathcal{A}$ is {\em generating}, that is $\mathcal{A}^\infty$ consists
of singletons. 
For every measurable set $U$ we will denote by $\mu_U$ with $\mu(U)>0$  the  measure conditioned on
 (the points starting in) $U$, that is  $\mu_U(A):=\frac{\mu(U\cap A)}{\mu(U)}$. As usual, for any collection of sets $\mathcal B$ we denote by $\sigma(\mathcal B)$ the smallest $\sigma$-algebra generated by $\mathcal B$.

 \subsection{Distribution of the number of visits in a fixed set $U$}
 
Initially, our  interest will be to characterise the distribution of the number of visits to sets with small measure. To this end, we define for any fixed set $U$ and for any $t>0$ the $\mathbb N$-valued random variable 
\begin{equation}\label{eq:W}
W=\sum_{i=0}^{t/\mu(U)}\ind_{U}\circ T^i\,,
\end{equation}
 which counts the number of visits to $U$ in the Kac's scaling $t/\mu(U)$. Although $W$ depends on $t$ and $U$, we do not explicit this dependence in the notation for the sake of simplicity.  Our first theorem gives an upper bound on the total variation distance\footnote{The total variation between two probability distributions $P$ and $Q$ on some measurable space $(\Omega,\mathcal F)$ is defined as $||P(\cdot)-Q(\cdot)||=\sup_{B\in\mathcal F}|P(B)-Q(B)|$.} between $\mu(W\in\cdot)$ and a {compound Poisson distribution $\tilde\nu$ with parameters $t\tilde\lambda_\ell ,\ell\ge1$}, that is, a probability distribution with generating function 
\begin{equation}\label{eq:poiscomp}
\varphi_{\tilde\nu}(z)=e^{\sum_{k\ge1}t\tilde\lambda_k\left(e^{zk}-1\right)}.
\end{equation}
Naturally,  the parameters $\tilde\lambda_\ell,\ell\ge1$ will depend on the dynamic. 

\vspace{0.2cm}

We stand in the world of $\phi$-mixing measures. 

\begin{definition}\label{def:phi}
We say  a $T$-invariant probability measure $\mu$ on $\Omega$ is {\em left $\phi$-mixing} 
with respect to the partition $\mathcal{A}$ if there exists a decreasing sequence
$\phi(k)\searrow0$ so that for every $n, m\ge1$, $ U\in\sigma(\mathcal{A}^n)$ and  
$V\in\sigma(\bigcup_{m=1}^\infty\mathcal{A}^m)$:
 \begin{equation}\label{def:phi_left}
 \left|\frac{\mu(U\cap T^{-n-k}V)}{\mu(U)}-\mu(V)\right|\le\phi(k).
 \end{equation}
 Similarly we say  $\mu$ is {\em right $\phi$-mixing} if under the same conditions
\begin{equation}\label{def:phi_right}
 \left|\frac{\mu(U\cap T^{-n-k}V)}{\mu(V)}-\mu(U)\right|\le\phi(k).
 \end{equation}
 \end{definition}

For any $i\ge1$ and any $1\le K\le i$, the variable counting the number of visits to $U$ at a distance less or equal to $K$ around $i$ is
\[
Z^{(K)}_i:=\sum_{j=i-K}^{i+K}\ind_{U}\circ T^j.
\]
{For $x\in \Omega$ denote by $A_j(x)$ the the unique atom of $\mathcal{A}^j$ which contains $x$.}
More generally for a set $U\subset \Omega$ we put for its outer $j$-cylinder approximation of $U$ ($j\ge1$) 
\begin{equation}\label{eq:outer_approx}
U^j:=A_j(U)=\bigcup_{A\in\mathcal{A}^j, A\cap U\not=\varnothing}A.
\end{equation}
{Similarly for $U\subset\Omega$ and and integer then, for $j<n$, we also define the $n$-right $j$-cylinder approximation by:
\begin{equation}\label{eq:tilde_approx}
\tilde{U}^j_n=T^{-(n-j)}A_j(T^{n-j}U)
=T^{-(n-j)}\bigcup_{A\in\mathcal{A}^j, A\cap T^{n-j}U\not=\varnothing}A.
\end{equation}
In the case when $U\in\sigma(\mathcal{A}^n)$ (union of $n$-cylinders) then we shall write below $\tilde{U}^j$
for $\tilde{U}^j_n$}\\
{(In Remark 5 we will give an example of a null set whose $n$-right $j$-cylinder approximation is the 
entire space for all $j<n/2$.)}
We are now ready to state our first main result where we denote by $\phi^1(\ell)=\sum_{j=\ell}^\infty\phi(j)$ 
the tail sum of $\phi$.
\begin{theorem}\label{theorem.compound.poisson}
Let $\mu$ be a $T$-invariant probability measure on $\Omega$ which is right $\phi$-mixing with  $\phi$ summable.
Then there exists a constant $C_1$ so that, for any measurable set $U\in \sigma(\mathcal{A}^n)$, any $t>0$
and any $K<t/\mu(U)$, we have 
\begin{equation}\label{eq:th1}
||\mu(W\in\cdot)-\tilde\nu_{K,U}||_{TV}\le C_1t\inf_{K<\Delta<t/\mu(U)} \!\left(K\frac{\phi(\Delta-n)}{\mu(U)}+\Delta\mu(U)
+\phi^1(K/2)+\sum_{j=K/2}^n\mu(U^j)\right),
\end{equation}
{($U^j$ as defined in \eqref{eq:outer_approx})} where 
 $\tilde\nu_{K,U}$ is the compound Poisson distribution with parameters
$t\tilde\lambda_\ell({K,U})$, $\ell\ge1$, where
\begin{equation}\label{eq:tildelambda}
\tilde\lambda_\ell{(K,U)}:=\frac{1}{\ell}\mathbb{E}\left(\ind_{Z^{(K)}_i=\ell}|\ind_{U}\circ T^i=1\right)\,\,,\,\,\,\forall i\ge K.
\end{equation}
{If we assume left $\phi$-mixing instead of right $\phi$-mixing, the same statement holds after 
replacing the $j$-cylinder approximations $U^j$ by the $n$-right $j$-cylinder approximations $\tilde U^j$ (as defined in \eqref{eq:tilde_approx}) of
$U\in\sigma(\mathcal{A}^n)$. }

\end{theorem}

\vspace{0.2cm}

We now consider the case in which we have a stronger kind of mixing called $\psi$-mixing. 

\begin{definition}
We say a $T$-invariant probability measure $\mu$ on $\Omega$ is {\em $\psi$-mixing} 
with respect to the partition $\mathcal{A}$ if there exists a decreasing sequence
$\psi(k)\searrow0$ so that for every $n, m\ge1$, $ U\in\sigma(\mathcal{A}^n)$ and  
$V\in\sigma(\bigcup_{m=1}^\infty\mathcal{A}^m)$:
 $$
 \left|\frac{\mu(U\cap T^{-n-k}V)}{\mu(U)\mu(V)}-1\right|\le\psi(k).
 $$
 \end{definition}

This stronger assumption naturally yields a stronger result. 
\begin{theorem}\label{corollary.compound.poisson_psi}
Let $\mu$ be a $T$-invariant probability measure on $\Omega$ which is $\psi$-mixing
where $\psi(j)\to 0$ as $j\to\infty$.
Then there exists a constant $C_1'$ so that, for any measurable set $U\in \sigma(\mathcal{A}^n)$, 
any $t>0$ and any $K<t/\mu(U)$, one has
\begin{equation}\label{eq:cor1}
\|\mu(W\in\cdot)-\tilde\nu_{K,U}\|_{TV}
\le
C_1't\inf_{K<\Delta<t/\mu(U)} \!\left(\psi(\Delta-n)+\Delta\mu(U)+\sum_{j=K/2}^n\mu(U^j)\right),
\end{equation}
where 
 $\tilde\nu_{K,U}$ is the compound Poisson distribution with parameters $t\tilde\lambda_\ell({K,U}),\ell\ge1$ given by~\eqref{eq:tildelambda}. 
 
{ By symmetry of $\psi$-mixing, the same inequality holds with $\tilde U^j$ instead of $U^j$ on the RHS. }
\end{theorem}

\subsection{Asymptotic distribution of the number of visits in a nested sequence $\{U_n\}_{n\ge1}$}\label{sec:asympt}

Now we will consider nested sequences of measurable sets $U_1\supset U_{2}\supset\ldots$ satisfying $\mu(U_n)\rightarrow0$. We will denote by $\Gamma$ the limiting null-set. Our interest is to study the convergence in distribution of 
\[
W_n:=\sum_{i=0}^{t/\mu(U_n)}\ind_{U_n}\circ T^i
\]
as $n$ diverges, for any $t>0$. 

Naturally, it is expected that, if in the Poisson compound approximations of the preceding theorems  the involved  parameters \eqref{eq:tildelambda} converge and we can further control the error terms, then we would have a Poisson compound distribution in the limit, parametrised by the limiting parameters. The statement of such a result  needs some more definitions on the entry/return time probabilities and the corresponding limiting quantities. 

For a subset $U\subset\Omega$ we define the first entry/return time
 $\tau_U$ by $\tau_U(x)=\min\{j\ge1: T^jx\in U\}$. Similarly we get higher order
 returns by defining recursively $\tau_U^\ell(x)=\tau_U^{\ell-1}+\tau_U(T^{\tau_U^{\ell-1}}(x))$
 with $\tau_U^1=\tau_U$. We also write $\tau_U^0=0$ on $U$. 
 
 We now come back to our nested sequence of sets $U_n,n\ge1$ and define (provided the limits exist) for $k,L,n\ge1$
\begin{align}\label{eq:alpha}
\nonumber\alpha_k(L,U_n)&:=\mu_{U_n}(\tau_{U_n}^{k-1}\le L<\tau_{U_n}^{k})\\
\nonumber\alpha_k(L)&:=\lim_{n\to\infty}\alpha_k(L,U_n)\\
\alpha_k&:=\lim_{L\to\infty}\alpha_k(L).
\end{align}

As promised, using Theorems \ref{theorem.compound.poisson} and  \ref{corollary.compound.poisson_psi}, and under proper further assumptions, we establish that the limiting distribution of the number of visits to the $U_n$'s is asymptotically compound Poisson. 
 
 \begin{theorem}\label{main.theorem}
Consider a nested sequence of sets $U_n\in \sigma(\mathcal{A}^n),n\ge1$, converging to a null-set $\Gamma$.
Suppose the  $T$-invariant probability measure $\mu$ satisfies: 
\begin{enumerate}
\item {either $\psi$-mixing, or right $\phi$-mixing with $\phi$ summable}, 
\item {there exists a vanishing sequence of positive real numbers $a_k,k\ge1$ such that $\sum_{i=k}^n\mu(U_n^i)\le a_k$ for all sufficiently large $n$'s,}
\item $\sum_{k=1}^\infty k^2\alpha_k<\infty$ (and naturally that the $\alpha_k$, $k\ge1$, exist, see \eqref{eq:alpha}).
\end{enumerate}
Then, for
%
%
%
%
 every $E\subset \mathbb{N}_0$ one has 
$$
\mu(W_n\in E)\longrightarrow\tilde\nu(E)
$$
as $n\to\infty$, where $\tilde\nu$ is the \emph{compound Poisson distribution with parameters 
$t\tilde\lambda_\ell$, $\ell\ge1$} and
\[
\tilde\lambda_k:=\alpha_k-\alpha_{k+1}.
\]
{If in assumption (1) we rather assume left $\phi$-mixing with $\phi$ summable, then we have to change $U^i_n$ to $\tilde U^i_n$ in (2), and the same statement applies . }
\end{theorem}

\section{Discussion of the results}\label{sec:discussion}

In this section, we list a series of remarks concerning the results presented in the previous section, together with some example illustrating these remarks.

\subsection{Concerning the assumptions}
Here we discuss the assumptions of the above theorems. 
\begin{itemize}
\item  It is classical in recurrence theory for dynamical systems to require some mixing conditions on the dynamic. Here we have two alternative assumptions which are not included one in the other. For Theorem \ref{main.theorem} for instance, we need either that the measure be $\psi$-mixing, or we require   right {(or left)} $\phi$-mixing with  polynomially decaying $\phi$. {The difference between assuming right or left $\phi$-mixing is made in order to handle the case of invertible maps (see Remark \ref{rem:bothsides} where, after the proof of Theorem \ref{theorem.compound.poisson}, this is explained)}. Plenty of examples satisfying these assumptions can be found in the literature \citep{bradley/2005,bradley2007introduction}. We will give some examples in Sections \ref{sec:examples} and \ref{sec:examples2}.

\item 
The assumption (2) of Theorem \ref{main.theorem} is necessary in our 
setting because in general, the $U_n$'s may be large unions of cylinders whose measures 
have to be controlled. It is clear that in the case where $\Gamma$ is a point, then our mixing 
assumptions automatically imply that $\mu(U_n)$ decays exponentially fast and thus 
satisfies the assumptions.

{If $U_n$ is the outer $n$-cylinder  approximation of $\Gamma$, then $U_n^j=U_j$ for any $j\le n$ and the condition simplifies to $\sum_{j\ge n}\mu(U_j)\to0$. }

\item Finally, we need that the $\alpha_k$'s exist and decay sufficiently fast so that $\sum_{k=1}^\infty k^2\alpha_k<\infty$. As we will explain, the existence/computation of the parameters $\alpha_k,k\ge1$ is not obvious in general, it is not granted by our mixing assumptions, and can only be, at most, guaranteed case by case. 
\end{itemize}

\subsection{Interpretation of the compound Poisson distribution}\label{rem:compPois} 

The definition  \eqref{eq:poiscomp} of the  Poisson compound distribution is not the most common in the literature. Let us explain that it indeed coincides with the classical definition.  Put $r:=\sum _\ell t\tilde \lambda_\ell$ and $\lambda_\ell:=t\tilde\lambda_\ell/r$. (Proposition~\ref{proposition.lambda} below will give conditions under which we have that $\lambda_\ell=\alpha_\ell-\alpha_{\ell+1}$.) With these quantities, we have
\[
\varphi_{\tilde\nu}(z)=e^{\sum_{k\ge1}t\tilde\lambda_k\left(e^{zk}-1\right)}=e^{r\sum_{k\ge1}\lambda_k\left(e^{zk}-1\right)}=e^{r\left(\sum_{k\ge1}\lambda_ke^{zk}-1\right)}.
\]
We recognise the moment generating function of the random variable $Z=\sum_{i=1}^NX_i$ in which $N\sim\textrm{Poisson}(r)$ and $X_i,i\ge1$ are i.i.d. integer valued r.v's with distribution
\[
P_X(\ell)
=\lambda_\ell
=\frac{\tilde\lambda_\ell}{\sum_k\tilde\lambda_k}\,,\,\,\ell\ge1.
\] 

When $\tilde\lambda_1=\lambda$ and $\tilde\lambda_k=0$, $k\ge2$, we obtain the straight Poisson distribution with parameter $t\lambda$. 

We can now make the relation with our results concerning the count of limiting
  returns to  sets with small measure. The interpretation of the Poisson random variable $N$ is that it gives the distribution
  of clusters which occur on a large
  timescale as suggested by Kac's formula. And the number of returns 
  in each cluster is given by the i.i.d.\ random variables $X_j$'s.
  These returns are on a fixed timescale and nearly independent of 
  the size of the return set as its measure is shrunk to zero. 

An important non-trivial compound Poisson distribution is the P\'olya-Aeppli
 distribution which happens when  the $X_j$'s are geometrically distributed with parameter $1-p$, 
 that is $P_X(k)=(1-p)p^{k-1},k\ge1$.
 
For instance, when $\tilde\lambda_\ell=(1-p)^2p^{\ell-1}$, the compound Poisson distribution with parameters 
$t\tilde\lambda_\ell,\ell\ge1$, is P\'olya-Aeppli since $\lambda_\ell=(1-p)p^{\ell-1}$. In this particular case 
we have moreover that $N\sim\text{Poisson}(t(1-p))$. This specific case will be called \emph{``P\'olya-Aeppli 
distribution with parameter $t(1-p)$''}.  This means in explicit form that 
$$
\tilde{\nu}(\{k\})=
e^{-(1-p)t}\sum_{j=1}^k\binom{k-1}{j-1}\frac{((1-p)^2t)^j}{j!}p^{k-j}.
$$

Several asymptotic distributions will appear along the paper, P\'olya-Aeppli or not, depending of the examples (and the setting).


\subsection{Relation to the extremal index in extreme value theory}\label{eq:EI}
Assuming that the $\alpha_k,k\ge1$ exist and vanish as $k$ diverges, we  have that $\sum_\ell\tilde\lambda_\ell=\sum_\ell(\alpha_\ell-\alpha_{\ell+1})$ telescopes to $\alpha_1$. This quantity, $\alpha_1:=\lim_{K\to\infty}\lim_{n\to\infty}\mu_{U_n}(K<\tau_{U_n})$, is called the {\em extremal index} and has a particular importance in extreme value theory \citep{freitas2013compound}.
 Under some circumstances \citep{abadi2020dynamical}, it is equal to the inverse of the mean cluster size. Indeed, according to Section~\ref{rem:compPois}, the expected cluster size is given by 
\[
\sum_{\ell}\ell\frac{\tilde\lambda_\ell}{\sum_k\tilde\lambda_k}
=\frac{\sum_{\ell}\ell\tilde\lambda_\ell}{\alpha_1}
=\frac{\sum_\ell\alpha_\ell}{\alpha_1}.
\]
It is explained in \cite{haydn2020limiting} (see for instance Theorem 2 and Remark 2 therein or see Proposition \ref{proposition.lambda} below) that, if $\sum_kk\sum_{\ell\ge k}\alpha_\ell<\infty$ (so in particular $\alpha_k$ exists and vanishes as $k$ diverges), then $\sum_\ell\alpha_\ell=1$ and we obtain the desired result $\frac{1}{\alpha_1}$ for the mean size of a cluster.

An important issue however is to know, for given dynamical systems, whether or not the limits appearing in all these quantities actually exist. We will investigate this question in Subsections \ref{sec:rege} and  \ref{sec:hoc} on some examples, and in Section \ref{sec:examples2} for the case of $g$-measures.


\subsection{Example 1: the House of cards process}\label{sec:hoc}

The house of cards process is a Markov chain on $\mathcal A=\{0,1,2,\ldots\}$ with transition matrix $Q$ parametrized by a sequence of $[0,1]$ real numbers $r_i,i\ge0$: 
\begin{equation}
Q(i,j)=\left\{
\begin{array}{ccc}
r_i&\text{if}&j=0\\
1-r_i&\text{if}&j=i+1.
\end{array}
\right.
\end{equation}
It has a stationary version if and only if $\sum_{i\ge1}\prod_{j=0}^{i-1}(1-r_j)<\infty$, which is the condition ensuring that the expecting distance between two consecutive occurrences of a $0$ is finite. In this case, the row vector $\pi$ satisfying $\pi Q=\pi$ is 
\[
\pi(k)=\pi(0)\prod_{i=0}^{k-1}(1-r_i)
\]
where
\[
\pi(0)=\frac{1}{1+\sum_{i\ge1}\prod_{j=0}^{i-1}(1-r_j)}.
\]

 For the stationary version of this Markov chain, we want to study the asymptotic distribution of the number of visits to runs of length $n$ above a threshold $l\ge1$. 

 We will use the stochastic process notation involving random variables, but in order to relate to the framework of Section \ref{sec:results}, we could let $\mu$ denote the measure on $\mathcal A^\mathbb N$ associated to the stationary process. This measure is $\sigma$-invariant, where $\sigma$ is the shift operator $\sigma:\mathcal A^{\mathbb N}\circlearrowleft$  defined through $(\sigma (x))_i=x_{i+1}$ for any $x=(x_0x_1x_1\dots)\in\mathcal A^{\mathbb N}$. We are interested in studying the statistics of visits of this symbolic system in 
$U_n=U_{n}(l)=\bigcap_{i=1}^{n} \sigma^{-i}[l,+\infty)$ as $n$ diverges for some fixed $l\ge1$.

Let $\{X_{i}\}_{i\ge0}$ be a stationary House of Cards Markov chain. By the Markov property, successive visits to $0$ parse the process into independent blocks.  Let us  denote 
\[
T:=\inf\{k\ge1:X_k=0\},
\]
and for any $i\ge0$
\[
q_i(k):=\mathbb P(T=k|X_0=i)=r_{i+k}\prod_{j=i}^{i+k-1}(1-r_j)
\]
the probability that the time elapsed until the next $0$, starting with $X_0=i$, be equal to $k$. 

Recalling the definition \eqref{eq:alpha} of $\alpha_{k}(L,U_n)$, we have
\[
\alpha_{k+1}(L,U_n)\ge\mathbb P(T=k|X_0\ge n+l)\mathbb P(X_i<n,i=1,\ldots,L-k|X_0=0),
\]
and
\begin{align*}
\alpha_{k+1}(L,U_n)\le &\,\,\mathbb P(T=k|X_0\ge n+l)\mathbb P(X_i<n,i=1,\ldots,L-k|X_0=0)\\& +\mathbb P(T<k|X_0\ge n+l)\mathbb P(T\ge n|X_0=0).
\end{align*}
Naturally, $T$ and $X_i$ being a.s. finite, we have that, as $n$ diverges,  $\mathbb P(X_i<n,i=1,\ldots,L-k|X_0=0)$ converges to 1 and $\mathbb P(T\ge n|X_0=0)$ converges to 0. We will prove below that, if $r_i\rightarrow r_\infty\in(0,1)$, then for any $k\ge1$ and any $l\ge1$
\begin{equation}\label{eq:convHoC}
\lim_n\mathbb P(T=k|X_0\ge n+l)=r_\infty(1-r_\infty)^k
\end{equation}
and therefore
\[
\alpha_{k+1}:=\lim_L\lim_{n\rightarrow\infty}\alpha_{k+1}(L,U_n)=\lim_{n\rightarrow\infty}\mathbb P_\pi(T=k|X_0\in U_n)=r_\infty(1-r_\infty)^k
\]
exists and decays exponentially fast in $k$, which grants Condition (3) of theorem \ref{main.theorem}.

Moreover, under the assumption $r_i\rightarrow r_\infty\in(0,1)$ we have that the Markov chain is Doeblin, and thus automatically exponentially {({right})} $\phi$-mixing \citep{bradley/2005}. This grants condition (1) of Theorem \ref{main.theorem}. Moreover, in our case, we have $U_n^j=U_j=\{X_i\ge l,i=0,\ldots,j-1\}$, thus
\[
\mu(U_n^j)=\sum_{n\ge l}\pi(n)\mathbb P(T\ge j|X_0=n)=\pi(0)\sum_{n\ge l}\prod_{i=0}^{j+n-1}(1-r_i)
\]
which is summable in $j$ since for any $\epsilon>0$, $r_i\ge r_\infty-\epsilon$ for large enough $i$'s, granting Condition (2) of Theorem \ref{main.theorem}.

Thus if $r_i\rightarrow r_\infty\in(0,1)$, we can apply Theorem \ref{main.theorem}, which gives us that the number of visits to $U_n$ is, asymptotically, P\'olya-Aeppli distributed with parameter $t(1-r_\infty)$. According to Subsection \ref{eq:EI},  the corresponding extremal index is $1/r_\infty$.

It only remains to prove the convergence \eqref{eq:convHoC}. Let us compute
\begin{align*}
\mathbb P(T=k|X_0\ge n+l)
&=\frac{\sum_{i\ge n+l}\mathbb P(T=k|X_0=i)\pi(i)}{\sum_{i\ge n+l}\pi(i)}\\
&=\frac{\sum_{i\ge n+l}r_{i+k}\prod_{j=i}^{i+k-1}(1-r_j)\pi(i)}{\sum_{i\ge n+l}\pi(i)}\\
&=\frac{\sum_{i\ge n+l}r_{i+k}\prod_{j=i}^{i+k-1}(1-r_j)\pi(0)\prod_{j=0}^{i-1}(1-r_j)}{\sum_{i\ge n+l}\pi(0)\prod_{j=0}^{i-1}(1-r_j)}\\
&=\frac{\sum_{i\ge n+k+l}r_i\prod_{j=0}^{i-1}(1-r_j)}{\sum_{i\ge n+l}\prod_{j=0}^{i-1}(1-r_j)}.
\end{align*}
By Stolz-Ces\`aro 
\begin{align*}
\lim_{n}\mathbb P(T=k|X_0\ge n+l)&=\lim_n\frac{r_{n+l+k}\prod_{j=0}^{n+l+k-1}(1-r_j)}{\prod_{j=0}^{n+l-1}(1-r_j)}\\
&=\lim_nr_{n+l+k}\prod_{j=n+l}^{n+l+k-1}(1-r_j)=r_\infty(1-r_\infty)^k
\end{align*}
as we said.

\subsection{Return and entry times}\label{sec:returnandentry}

An important task, in order to apply Theorem \ref{main.theorem}, is to prove that the involved limiting quantities exist and  to compute the sequence $\tilde\lambda_k,k\ge1$ (or, equivalently, $\lambda_k,k\ge1$, see Subsection~\ref{rem:compPois}), parameter of the asymptotic compound Poisson distribution. Here we give some alternative ways to prove these facts, by defining other quantities related to $\tilde\lambda_k$, $k\ge1$, which are eventually easier to handle. 

Let us define $\hat\alpha_\ell(K,U_n):=\mu_{U_n}(\tau_{U_n}^{\ell-1}\le K)$ and  assume that
  $\hat\alpha_\ell(K)=\lim_{n\to\infty}\mu_{U_n}(\tau_{U_n}^{\ell-1}\le K)$
  exist for $K$ large enough.
  Since $\{\tau_{U_n}^{\ell+1}\le K\}\subset \{\tau_{U_n}^{\ell}\le K\}$ we get that
  $\hat\alpha_\ell(K)\ge\hat\alpha_{\ell+1}(K)$ for all $\ell$ and in particular $\hat\alpha_1(K)=1$.
  By monotonicity the limits $\hat\alpha_\ell=\lim_{K\to\infty}\hat\alpha_\ell(K)$
  exist and satisfy $\hat\alpha_1=1$ and $\hat\alpha_\ell\ge\hat\alpha_{\ell+1}$ for any $\ell\ge1$.
  Now assume that moreover the limits $p_i^{(\ell)}=\lim_{n\to\infty}\mu_{U_n}(\tau_{U_n}^{\ell-1}=i)$
  of the conditional size of the level sets of the $\ell^{\text{th}}$ return time $\tau_{U_n}^\ell$
  exist for $i\ge0$ (clearly $p_i^{(\ell)}=0$ for $i\le\ell-2$). 
  According to Lemma 1 in \cite{haydn2020limiting} one has, for $\ell\ge2$,
 \[
  \hat\alpha_\ell=\sum_ip_i^{(\ell)}.
\]
We also have, by definition, that $\alpha_\ell=\hat\alpha_\ell-\hat\alpha_{\ell+1},\ell\ge1$. So the existence of the $\alpha_\ell$'s  is granted once the $\hat\alpha_\ell$'s exist, moreover, according to what we just said
\[
\alpha_\ell=\sum_i(p_i^{(\ell)}-p_i^{(\ell+1)})\,,\,\,\ell\ge2.
\]
This  relation also holds for $\ell=1$. To see this, first recall that
$$
\alpha_1=\lim_{K\to\infty}\lim_{n\to\infty}\mu_{U_n}(K<\tau_{U_n})
$$
 and observe  that $p_0^{(1)}=1$ and $p^{(1)}_i=0, i\ge1$.
 It follows that  we can write
$\sum_i(p_i^{(1)}-p_i^{(2)})=1-\sum_ip^{(2)}_i=1-\hat\alpha_2=\alpha_1$.

%
%
%
%
%

Finally, we state without proof  the following result which was proven in~\cite{haydn2020limiting}, and which gives an important characterization of $\lambda_\ell$ under some conditions.

\begin{proposition}[\cite{haydn2020limiting}]\label{proposition.lambda}
Let $U_n\subset\Omega$ be a nested sequence so that $\mu(U_n)\to0$ as $n\to\infty$. 
Assume that the limits $\hat\alpha_\ell(L)=\lim_{n\to\infty}\hat\alpha_\ell(L,U_n)$ exist for $\ell=1,2,\dots$
and $L$ large enough.
Assume $\sum_\ell\ell\hat\alpha_\ell<\infty$, then 
$$
\lambda_k=\frac{\alpha_k-\alpha_{k+1}}{\alpha_1}
$$
where $\alpha_k=\hat\alpha_k-\hat\alpha_{k+1}$.
In particular the limit defining $\lambda_k$ exists.
\end{proposition}

%
%

 
\subsection{Example 2: Regenerative processes}\label{sec:rege}

We recall that a a stochastic process $\{X_i\}_{i\ge0}$  is a regenerative process if there exist random times $T_1<T_2<\ldots$ such that the sigma fields $\sigma(X_{T_{n}}^{T_{n+1}-1}),n\ge1$ are i.i.d. and independent of $\sigma(X_0^{T_{1}-1})$. So the model is completely defined if we specify the distribution of $X_0^{T_{1}-1}$, and $X_{T_{1}}^{T_{2}-1}$. 
Here we consider a particular case in which  these vectors belong to $\bigcup_{a\in\mathcal A}\bigcup_{k\ge1}a^k$ where $a^k$ denotes the vector $(a,\ldots,a)$ of $k$ times the same symbol $a$ concatenated, and $\mathcal A\subset \mathbb N$. In other words, the independent blocks are filled up with only one symbol as follows
\[
X_0^\infty=\underbrace{X_0\ldots X_0}_{\text{$T_{1}$ times}}\underbrace{X_{T_1}\ldots X_{T_1}}_{\text{$T_2-T_1$ times}}\ldots \underbrace{X_{T_n}\ldots X_{T_n}}_{\text{$T_{n+1}-T_{n}$ times}}\ldots
\]
Specifically, we consider that for any $a\in\mathcal A$ and $k\ge1$
\[
\mathbb P(X_{T_{1}}^{T_{2}-1}=a^k)=p(a)q_a(k)
\]
where $\sum_ap(a)=1$ and, for any $a\in\mathcal A$, $\sum_kq_a(k)=1$. A way to interpret the above formula is in a two-steps procedure. First we choose the symbol $X_{T_1}=a$ independently of everything, with probability $p(a)$, and next, we choose the size $k$ of the block, with probability $q_a(k)$. 
This is a particular instance of Semi-Markov process  \citep{janssen2006applied,cinlar2013introduction}. In particular, it is well-known that there exists a stationary version of the process if and only if the expectation of the blocks is finite, that is
\[
\nu=\sum_{a\in\mathcal A}p(a)\nu_a:=\sum_{a\in\mathcal A}p(a)\sum_{k\ge1}kq_a(k)<\infty
\]
where $\nu_a$ is the expectation of the blocks of symbols $a$. Another known fact is that, for the process to be stationary, the distribution of $X_0^{T_{1}-1}$ must be 
\[
\mathbb P(X_0^{T_{1}-1}=a^k)=\bar p(a)\bar q_{a}(k)\,\,,\,\,\,a\in\mathcal A,k\ge1
\]
where 
\[
\bar p(a)=\frac{p(a)\nu_a}{\nu}\,\,\,\,\,\text{and}\,\,\,\,\,\,
\bar q_a(k)=\frac{\sum_{l\ge k}q_a(l)}{\nu_a}. 
\]

We want to study the distribution of the number of visits to states larger or equals to $n$ when $n$ gets large (we assume that $\mathcal A$ is countably infinite).

As for the House of Cards Markov chain, we can relate to the framework of Section \ref{sec:results} by considering the symbolic measures space $(\mathcal A^{\mathbb N},\mathcal B,\mu,\sigma)$, and this time, we consider the nested sets $U_n=[n,+\infty),n\ge1$. 

The regenerative structure was also present in the House of Cards Markov chain, since visits to $0$ cut the realisation into independent blocks. However, regenerative processes need not be Markovian. The first step, if we want to use Theorem \ref{main.theorem} is to investigate the mixing properties of this model. 
\begin{proposition}\label{prop:rege_phi}
For the regenerative process described above, {inequality \eqref{def:phi_right}} holds for
\[
\phi(k)= 2\sup_{a\in\mathcal A}\sum_{i> k}\bar q_{a}(i).
\]
\end{proposition}
The proof of this proposition is not difficult but since we did not find it in the literature, we do it in Section \ref{sec:remaining_results}.

An interesting case is the \emph{Smith example}, in which $q_a(a+1)=\frac{1}{a}=1-q_a(1)$. This example was used by \cite{haydn2020limiting} as a case in which Proposition \ref{proposition.lambda} cannot be used, because they proved that $\hat\alpha_k=\frac{1}{2}$ for any $k$. In any case, we cannot use Proposition \ref{prop:rege_phi} neither, and thus we cannot conclude on the statistics of this model here. 

To simplify the presentation, suppose now that  $q_a(l)=q(l)$ for any $l\ge1$ independently of $a$, in which case $\bar q_a=\bar q$ for any $a$ also. According to Proposition \ref{prop:rege_phi}, Condition (1) of Theorem \ref{main.theorem} is granted if the probability distribution $\bar q$ has first moment. 

On the other hand, it is not too complicated to see that $\alpha_k(L,U_n)$ (see \eqref{eq:alpha}) is close to the probability that the starting block, which is a block of a symbol in $U_n$ (since we are conditioned on starting in $U_n$), equals $k$. Indeed, using a similar reasoning as the one used for the house of cards Markov chain \cite[Section 8.2]{haydn2020limiting}, we get
\[
\alpha_k(L,U_n)=\bar q(k)+\mathcal O(L\mu(U_n))
\]
from which it follows that
\[
\alpha_k=\bar q(k)= \frac{\sum_{l\ge k}q(l)}{\sum_kkq(k)}
\]
since in the present case $q_a(k)=q(k)$ for any $k$. So condition (3) of Theorem \ref{main.theorem} is granted if we assume that $\bar q$ has finite second moment, that is, $\sum k^2\bar q(k)<\infty$. This in turns is granted if $q$ has third moment, since (we put $\nu_a=\nu$ for any $a$)
\[
\sum_kk^2\bar q(k)=\frac{1}{\nu}\sum_kk^2\sum_{l\ge k} q(l)\le\frac{1}{\nu}\sum_k\sum_{l\ge k} l^2q(l)=\frac{1}{\nu}\sum_kk^3q(k)<\infty.
\] 
{In order to ensure condition (2), we will assume that $n\sum_{i\ge n}p(i)\rightarrow0$ as $n$ diverges, which is the case for instance if the probability distribution $p(a),a\in\mathcal A$ has first moment. Then, since in our case $U_n^i=U_n$ for any $i\le n$, we have for any $1\le K$ and sufficiently large $n$
\[
\sum_{i=K}^n\mu(U_n^i)\le n\mu(U_n)=n\sum_{i\ge n}p(i)\le K\sum_{i\ge K}p(i).
\]
Thus taking $a_k= k\sum_{i\ge k}p(j)$, the assumption (2) of Theorem \ref{main.theorem} also holds. }

We conclude that, if $p(a),a\in \mathcal A$ has first moment and $q(k),k\ge1$ has third moment, then the number of visits to $U_n$ is, asymptotically compound Poisson with parameter $t\tilde\lambda_k,k\ge1$ where
\[
\tilde\lambda_k=\frac{q(k)}{\sum_kkq(k)}.
\]

As explained in Subsection \ref{eq:EI}, $\alpha_1=\sum_k\tilde\lambda_k$ which in this case gives $1/\sum_kkq(k)$. Thus the extremal index is $1/\alpha_1=\sum_kkq(k)$ which is the expected block size. 


\subsection{Number of visits around a point}\label{sec:examples}

Determining the limiting distribution of the number of visits when the limiting target set is a point and the nested sequence of sets is a sequence of cylinders containing the point is a classical question in the literature of recurrence theory \citep{haydn2013entry}. It is however a nice way to illustrate Theorem \ref{main.theorem}.

Suppose that the mixing conditions of Theorem \ref{main.theorem} are satisfied by the dynamic under consideration. 
Fix a point $x\in\Omega$ and for any $n\ge1$ consider the $n$-cylinder $A_n(x)$, that is, the unique atom of $\mathcal{A}^n$ containing $x\in\Omega$. In the notation of the previous section, we let $U_n=A_n(x)$ (that is $\Gamma=\{x\}$), and we ask what  is the asymptotic distribution of visits to $U_n,n\ge1$. It is well-known that in this case there is a dichotomy according to whether $x$ is aperiodic, in which case we have a Poisson distribution, or periodic, in which case we have a Polya-Aeppli distribution instead. Illustrating how to use our results in this simple example will be the opportunity to clarify technical details concerning notation and some involved limiting quantities.

Initially, for any measurable set $U\subset\Omega$ we write $\tau(U)=\inf_{y\in U}\tau_U(y)$ for 
 the {\em period} of $U$. In other words,  $U\cap T^{-j}U=\varnothing$ for 
 $j=1,\dots,\tau(U)-1$ and $U\cap T^{-\tau(U)}U\not=\varnothing$.
The proof of the following lemma is direct and  can be found for instance in \cite{haydn2020limiting}. 
 
 \begin{lemma} Let $\mathcal{A}$ be a (finite) generating partition of $\Omega$.
 Then the sequence $\tau(A_n(x))$, $n=1,2,\dots$ is bounded if and only if
 $x$ is a periodic point.
 \end{lemma}



We start with the case where $x$ is a periodic point,
 with minimal period $m$, say. Let us compute the values $\lambda_\ell$.  For $n$ large enough one has $\tau(A_n(x))=\tau_\infty=m$
 and therefore $A_n(x)\cap\{\tau_{A_n(x)}=m\}=A_n(x)\cap T^{-m}A_n(x)=A_{n+m}(x)$. Moreover  
 $p_i^{(\ell)}=\lim_{n\rightarrow\infty}\mu_{A_n(x)}(\tau^{\ell-1}_{A_n(x)}=i)=0$ for $i<m$. 

 Assume the limit
\begin{equation}\label{eq:limit}
p:=p^{(2)}_m=\lim_{n\to\infty}\frac{\mu(A_{n+m}(x))}{\mu(A_n(x))}
\end{equation}
  exists, then one also has more generally 
 $$
 p^{(\ell)}_{(\ell-1)m}=\lim_{n\to\infty}\frac{\mu(A_{n+(\ell-1) m}(x))}{\mu(A_n(x))}=p^{\ell-1}.
 $$

 All other values of $p^{(\ell)}_i$ are zero, that is $p^{(\ell)}_i=0$ if $i\not=(\ell-1)m$.
 Thus $\hat\alpha_\ell=p^{(\ell)}_{(\ell-1)m}=p^{\ell-1}$
  and consequently
 $$
 \alpha_\ell=\hat\alpha_\ell-\hat\alpha_{\ell+1}=(1-p)p^{\ell-1}
 $$
  which is a geometric distribution and in particular implies that $\sum_kk^2\alpha_k<\infty$, meaning that condition (3) of Theorem \ref{main.theorem}. Moreover, in the present case, we have $U_n^i=U_n$ for any $i\le n$, and since our mixing assumptions imply that the measure of cylinders decays exponentially fast in $n$, it follows that Assumption (2) of Theorem \ref{main.theorem} is automatically granted. So by Theorem \ref{main.theorem}, we conclude that the random variable $W$ has P\'olya-Aeppli distribution  with parameter $t(1-p)$ (see Subsection \ref{rem:compPois}).  
  
We now consider the case of a non-periodic point $x$. In this case, the increasing sequence $\tau(A_n(x))$ goes to infinite
as $n\to\infty$.  Note that 
$\mu_{A_n(x)}(\tau_{A_n(x)}\le K)=0$ for all $n$ large enough so that $K<\tau(A_n(x))$.
Hence $\hat\alpha_2(K)=0$ for all $K$ which implies that $\hat\alpha_2=0$
and consequently $\hat\alpha_\ell=0$ for all $\ell\ge2$.
Consequently in this case the extremal index is $\alpha_1=1-\hat\alpha_2=1$ and $\alpha_k=0,k\ge2$ so that $\tilde\lambda_1=1$ and $\tilde\lambda_k=0,k\ge2$ and therefore $W$ is Poisson$(t)$ distributed.
%


\section{The case of $g$-measures}\label{sec:examples2}

Let $\mathcal{A}=\{1,2,\dots,M\}$ be a finite alphabet and 
\[
\Sigma_B=\left\{x\in\mathcal{A}^{\mathbb{N}}: B_{x_i,x_{i+1}}=1,\,\forall i\ge1\right\}\subset\Sigma:=\mathcal A^\mathbb N
\]
where $B$ is an aperiodic and irreducible $M\times M$ matrix of $0$'s and $1$'s. Let $\mathcal F$ denote the Borel $\sigma$-algebra of $\Sigma_B$. 
For any finite string $a_1^n$ (shorthand notation for $(a_1,\ldots a_n),a_i\in\mathcal A$) of symbols of $\mathcal A$, 
we let  $[a_1^n]:=\{x\in\Sigma_B:x_i=a_i,i=1,\ldots,n\}$ denotes the corresponding cylinder set. 
The Borel $\sigma$-algebra $\mathcal F$ is generated by the cylinder sets. The shift operator $\sigma:\Sigma_B\circlearrowleft$ 
 defined through $(\sigma (x))_i=x_{i+1}$ for any $x=(x_1x_2x_3\dots)\in\Sigma_B$ is called sub-shift of finite type. 
 
A measurable function $g:\Sigma_B\rightarrow[0,1]$ satisfying 
\begin{equation}\label{eq:g}
\sum_{y:\sigma(y)=x}g(y)=\sum_{a\in\mathcal{A}}g(ax)=1
\end{equation}
 for any $x\in\Sigma_B$ is called a \emph{$g$-function}. Let $\mathcal{L}_g$ be the associated transfer operator given by 
$$
\mathcal{L}_gf(x)=\sum_{y:\sigma(y)=x}g(y)f(y)=\sum_{a\in\mathcal{A}}g(ax)f(ax),
$$
for functions $f:\Sigma_B\to\mathbb{R}$, where $ax$ is the concatenation
of the symbol $a$ with the sequence $x$ (if admissible).
A \emph{$g$-measure} is a probability measure satisfying $\mathcal{L}_{g}^*\mu=\mu$ \citep{keane/1972} 
where $\mathcal{L}_{g}^*$ is the dual of $\mathcal{L}_g$. This is equivalent \citep{ledrappier} to $\mu$ 
being $\sigma$-invariant and satisfying
\[
\mathbb{E}_{\mu}(\ind_{[a]}|\mathcal{F}_2^{\infty})(x)=g(a\,\sigma (x)),
\]
 for any $a\in\mathcal A$ and $\mu$-almost every $x\in\Sigma_B$. Yet another equivalent way  is to define $\mu$ is through the variational principle, 
\[
\mu\in\text{argmax}\left\{h_{\nu}+\int\log g\,d\nu:\nu\,\text{ is $\sigma$-invariant} \right\}
\]
where $h_\nu$ denotes the Kolmogorov-Sinai entropy. The maximum
 of the quantity above is called the topological pressure of $\log g$ and denoted $P(\log g)$. It turns out that, since $\sum_{y:\sigma(y)=x}g(y)=1$,  we have $P(\log g)=0$. 

All the above can be stated in the framework of \emph{equilibrium states} for a real function 
$\varphi$ on $\Sigma_B$. This can be done simply by substituting $g$ by $e^\varphi$, except for 
the restriction \eqref{eq:g} which is put in the $g$-measure context to give a stochastic process flavour. 
We refer to \cite{ledrappier, walters/1975} for the proofs of all the above equivalences and further
 details on the variational principle for generic potentials and $g$-functions.

An important characterisation of the regularity of $g$ is its variation of order $k\ge1$
\begin{equation}\label{eq:variation}
\text{var}_k\,g:=\sup\{|g(x)-g(y)|:x_1^k=y_1^k\}.
\end{equation}
The convergence $\text{var}_k\, g\rightarrow0$ is equivalent to uniform continuity in the product topology. In this case,  $\mathcal{L}_{g}^*\mu=\mu$ has at least one solution \citep{keane/1972}.  Under the stronger assumption that 
$\sum_k\text{var}_k\, g<\infty$ and $g>0$, there is a unique $g$-measure specified by $g$ \citep{ledrappier} and it enjoys $\psi$-mixing \citep[see the proof of Theorem 3.2 therein]{walters/1975}.

\subsection{Visits close to a periodic point} In Subsection \ref{sec:examples}, we considered the case of visits around a point through cylinders. Concerning periodic points of minimal period $m$, the existence of the limit 
\begin{equation}\label{eq:limit2}
p:=p^{(2)}_m=\lim_{n\to\infty}\frac{\mu(A_{n+m}(x))}{\mu(A_n(x))}
\end{equation}
 was assumed in order to conclude the asymptotic distribution of the number of visits close to the point. Here we consider this question in the case of $g$-measures. We have the following proposition.

\begin{proposition}\label{prop:polya_gmeasure}
Consider a $g$-measure $\mu$ and a point $x\in\Sigma_B$ of prime period $m\ge1$. If $g$ is continuous at the set of points $\{\sigma^i(x),i=0,\ldots,m-1\}$, then the limiting parameter $p$ defined through \eqref{eq:limit2} exists and 
is given by  $\prod_{i=0}^{m-1}g(\sigma^i(x))$. So in particular, if $\var_k\,g$ vanishes, 
the limiting parameter exists for \emph{any} periodic point.  If moreover $\sum_k\var_k\, g<\infty$ and $g>0$, 
the limiting distribution of the number of visits around $x$  has P\'olya-Aeppli distribution  with parameter $t(1-p)$.
\end{proposition}

\begin{remark}
Let us mention that the existence of the limit was  proven and computed for Axiom A by \cite{pitskel/1991}.
\end{remark}


Let us now consider a specific class of $g$-measures, called \emph{renewal measures}.  Consider the space $\Sigma=\{0,1\}^{\mathbb N}$ and for any $x\in \Sigma$ let $\kappa(x):=\inf\{n\ge0:x_{n+1}=1\}$ count the number of $0$ until the first occurrence of a $1$ in $x$. Now take a sequence of $[0,1]$-valued real numbers $q_i,i\ge0$ and define the function  $\bar g$ by $\bar g(1x)=q_{\kappa(x)}$. A $g$-measure corresponding to $\bar g$ exists under some technical assumptions on the sequence $q_i,i\ge1$, which are automatically granted if we assume that $q_i\ge\epsilon$ for any $i\ge1$. 

\begin{proposition}\label{prop:polya_renewal}
Consider a renewal measure with sequence of parameters $\{q_i\}_{i\ge1}$ which  satisfies $q_i\in[\epsilon,1-\epsilon]$ for some $\epsilon>0$. Then, for any periodic point $x\ne 0^{\infty}$, the limit $p$ defined by \eqref{eq:limit} 
exists and the limiting distribution of the number of visits around $x$  has P\'olya-Aeppli distribution  with parameter $t(1-p)$.

 The same occurs for visits around $0^{\infty}$ if, and only if, $q_i,i\ge1$ converge. 
\end{proposition}

The main interest of this example lays in the fact that it is $\phi$-mixing.

\subsection{Existence of the extremal index}

Due to its relation to the so-called extremal index, the question of (non-)existence of the limit \ref{eq:limit2} was investigated recently in \cite{abadi2019clustering}. According to Proposition \ref{prop:polya_gmeasure}, vanishing variation guarantees existence of the extremal index at any periodic point. So if we want to characterise systems for which the extremal index does not exist, we have to get out of the classical setting of $g$-measures in which $g$ is assumed uniformly continuous. This is what we discuss now. 

First, let us observe that Theorem 3.2 of \cite{abadi/cardeno/gallo/2015} completely solved the question of the existence (and computation) of the limiting parameters in the case of renewal measures. Something interesting which is shown therein is that the  renewal measure  provides a simple situation in which $p$ does not exist although the measure enjoys good mixing. This is the case if we take  $q_{\kappa(x)}=\epsilon_1$ if $\kappa(x)$ is odd and $\bar g(x)=\epsilon_2$ otherwise: the limit $\lim\frac{\bar\mu([0^{n+1}])}{\bar\mu([0^n])}$ does not exist (see also Proposition \ref{prop:polya_renewal} above). With this choice of parameters, the measure is  $\phi$-mixing with exponentially decaying rate $\phi$, but we easily see that $\bar g$ has a discontinuity (with respect to the product topology) at the point $0^\infty$. 

The first information we get from this example is that good mixing properties are not enough to ensure existence of the extremal index, and that this existence is perhaps related to the continuity properties of the $g$-function. Technically, the discontinuity of $\bar g$ at $0^{\infty}$ is an \emph{essential discontinuity}, borrowing the terminology used in the context of statistical physics \citep{fernandez2005gibbsianness}. This is a discontinuity which cannot be removed by changing function $\bar g$ on a null $\bar\mu$-measure subset of $\Sigma$ \citep{ferreira2020non}. 

More generally, for a $g$-measure $\mu$ with $g$-function $g$, the non-existence of the limit $\lim_n{\mu([a_1^{n}])}/{\mu([a_1^{n-1}])}$ implies that $g$ has an essential discontinuity at $y_1^\infty$, but the converse is not necessarily true. So even  when we focus on the easier  case of points of period $1$, the non-existence of the extremal index implies on an essential discontinuity of $g$ at $a^{\infty}$, but the converse is not true in general. This non-equivalence is spectacularly clear with the following example.

\subsection{The Furstenberg $\&$ Furstenberg example} As far as we know, the following example is due to \cite{furstenberg1960stationary} (see Chapter 3.12 therein). On $\mathcal{A}^\mathbb{N}=\{-1,+1\}^\mathbb{N}$, take the product measure $\mu$ with marginal $\mu([+1])=\epsilon=1-\mu([-1])$. Next, consider the function $\Pi:\{-1,+1\}^\mathbb{N}\rightarrow\{-1,+1\}^\mathbb{N}$ defined through $(\Pi(x))_i=x_ix_{i+1}$, the product of two consecutive coordinates. The measure $\nu:=\mu\circ\Pi^{-1}$ has a $g$-function which is essentially discontinuous everywhere \citep{verbitskiy/15,ferreira2020non}. Let us now write down its $g$-function $g^\star$.

For any fixed $j\in \mathbb N$, when the limit exists, let
\[
d(x_j^\infty)=\lim_{n\to\infty}\frac{\#\{j\le i\le n, x_i=-1\}}{n}
\]
denote the asymptotic density of $-1$ in the sequence $x_j^\infty$  with $x_i\in\{-1,+1\},i\ge j$.  For any fixed $y\in\{-1,+1\}^\mathbb{N}$ the preimage set $\Pi^{-1}(y)$ contains two elements, that we denote by $x^+(y)$ for the one starting by $+1$ and $x^-(y)$ for the one starting by $-1$.  Now, let
\[
G:=\{y\in\{-1,+1\}^{\mathbb N}:d(x^+(y))=\epsilon\,\,\text{or}\,\,d(x^-(y))=\epsilon\}.
\]
Then by the law of large numbers for the product measure $\mu$, we have that $\nu(G)=1$. It is proved in \cite{ferreira2020non} that for any $y\in G$, 
\begin{equation}\label{eq:conv_furst}
\frac{\nu([1y_2^n])}{\nu([y_2^n])}\rightarrow \left\{\begin{array}{ccc}
		\epsilon&\text{if}&d(x^-(y))=\epsilon\\
		1-\epsilon&\text{if}&d(x^+(y))=\epsilon.
	\end{array}\right.
\end{equation}
This defines $g^\star(1\sigma(y))$ for $y\in G$ by the Martingale Theorem. For $y\in G^c$, $g^\star(1\sigma(y))$ may be defined arbitrarily, as this set has null $\nu$-measure and the choice will not affect the conclusion of everywhere essential discontinuity. 

The following simple result proves that, despite of the terrible (dis)continuity properties of $g^\star$, the limiting quantities needed to apply our theorems exist. 
\begin{proposition}\label{prop:furst}
Consider the measure $\nu$ with $g$-function $g^\star$ as defined above, with $\epsilon\ne1/2$. Then, for any periodic point $x$ of prime period $m\ge1$, the limiting quantity $p$ defined by \eqref{eq:limit} exists
and the limiting distribution of the number of visits around $x$  has P\'olya-Aeppli distribution  with parameter $t(1-p)$.
\end{proposition}

To conclude on this example, we observe that the value $p$ is explicitly computed in the proof of the proposition (see \eqref{eq:limitingFurst}).

\subsection{Temporal synchronisation for $g$-measures}\label{sec:temporal_sync}
Consider $m$ $g$-measures $\mu_1,\ldots,\mu_m$  on $\Sigma_B$, respectively with functions $g_1,\ldots,g_m$, and define the product measure $\hat\mu=\mu_1\otimes\ldots\otimes\mu_m$  on $\Omega=\Sigma_B^m$ or even $\Sigma^m$. 
Let 
$\hat\sigma:\Omega\circlearrowleft$ be the shift map on the product space. 
For any $n\ge1$, $S_n=\bigcup_{A\in\mathcal{A}^n}A^m$ is the $n$-cylinder
neighbourhood of the diagonal $\Delta=\{(x,\ldots,x): x\in\Sigma_B\}\subset \Sigma_B^m$ (see Figure \ref{fig2} for a picture with $m=2$). 
Observe that 
\[
\hat \sigma^{-i}S_n=\{(x^{(1)},\ldots,x^{(m)}):x^{(1)}_i\ldots x^{(1)}_{n+i-1}=x^{(j)}_i\ldots x^{(j)}_{n+i-1},j=2,\ldots,m\},
\]
that is, a visit in $S_n$ can be interpreted as a synchronisation lasting $n$ time units of the symbols of the dynamical systems, therefore justifying the name ``temporal synchronisation''. 

A natural first problem is, in the above ``uncoupled'', or ``non-interacting'' setting ($\hat\mu$ is the product measure), to study the distribution of the number the visits to $S_n$ as $n$ diverges, that is, visits to longer and longer synchronised pieces of orbits. However, it would be even more interesting to study the same question for interacting $g$-measures.  In full generality, for any $m\ge2$, any $g$-measure on the product space $\Omega=(\mathcal A^m)^\mathbb N$ can be considered the coupling of the $m$ coordinates $g$-measures. That is, we see $\Omega$ as $\Sigma^m$ instead of seeing it as $(\mathcal A^m)^\mathbb N$. In other words, in this general setting, we are studying the synchronisation of the coordinates.  

Theorem \ref{theo:tempsync} below is stated with this abstract approach because it is more general, and next, Corollary \ref{coro:sync_indep} will specialise to the non-interacting case. Finally, we will rapidly discuss two explicit ways to make $g$-measures interact. 

Let $W_n$ count the number of  synchronisations on the Kac scaling
\[
W_n:=\sum_{i=0}^{t/\hat{\mu}(S_n)}\ind_{S_n}\circ \hat\sigma^i.
\]

%
%
\begin{theorem}\label{theo:tempsync}
As before let $B$ be irreducible and aperiodic. Then
\begin{enumerate}
\item On the product space  $\Omega=\Sigma^m$, for some $m\ge2$, let $\hat\mu$ be a $\hat\sigma$-invariant $\hat{g}$-measure. Assume that $\hat g>0$ has summable  variation.

Then we have that 
\[
p:=\lim_n\frac{\hat\mu(S_n\cap \hat \sigma^{-1}S_n)}{\hat\mu(S_n)}
\]
exists and $\hat\mu(W_n\in\cdot)$ converges to a P\'olya-Aeppli distribution with parameter $t(1-p)$.
\item On the product space  $\Omega=\Sigma_B^m$, for some $m\ge2$, let $\hat\mu$ be a 
$\hat\sigma$-invariant $\hat{g}$-measure. Assume that the function $g^\Delta:\Sigma_B\to\mathbb{R}$ has exponentially decaying 
variations, where $g^\Delta(x)=\hat{g}(x,x,\dots,x)$. 

Then, $\hat\mu(W_n\in\cdot)$ converges to a P\'olya-Aeppli distribution with parameter $t(1-p)$ where 
$p=e^{P\left(\log g^\Delta\right)}<1$ with $P(\log g^\Delta)$ being the topological pressure of $\log g^\Delta$. 
\end{enumerate}\end{theorem}

\begin{remark}\label{rem:markov_pressure0}
In the particular case in which $\hat g$  only depends on the two first coordinates, $\hat\mu$ is a Markov chain with matrix $\hat Q(a,b):=\hat g(x)$ for any $x \in \Sigma^m$ such that $x_1=b$ and $x_2=a$. In this case, $P(\log g^{\Delta})$ is $\log \rho$ where $\rho$ is the largest positive eigenvalue of the matrix $Q^\Delta$. 
\end{remark}

We have the following direct corollary of item (2) of the preceding theorem.

\begin{corollary}\label{coro:sync_indep}
On the subshift space $\Sigma_B$, consider $m\ge2$ independent $g$-measures $\mu_i,i=1,\ldots,m$, 
with $g$-functions $g^{(i)}$ satisfying $g^{(i)}>0$ and having exponentially vanishing variation. 

Then, $\hat\mu(W_n\in\cdot)$ converges to a P\'olya-Aeppli distribution with parameter $t(1-p)$ where 
$p=e^{P\left(\sum_{i}\log g^{(i)}\right)}$, with $P(\sum_{i}\log g^{(i)})$ being the topological pressure of $\sum_{i=1}^m\log g^{(i)}$. 
\end{corollary}

The proof of this corollary is direct as, in the uncoupled case,  $\hat{g}(x_1, x_2, \dots, x_m)=\prod_{i=1}^mg^{(i)}(x_i)$ and therefore $\hat g$ inherits the regularity and mixing properties of the $g^{(i)}$'s. 

\begin{remark}
The R\'enyi entropy of order $q\in\mathbb R$ is defined as the limit
\[
\mathcal R_\mu(q)=-\lim_{n\to\infty}\frac{1}{qn} \log \sum_{[x_1^n]}  \mu([x_1^n])^{q+1}
\]
when it exists.
Proposition 7 in \cite{abadi2019complete} states that it exists and equals $-\frac{P((1+q)\log g)}q$ as long as $g$ is continuous. So in the case of the synchronization of $m$ independent copies of the same $g$-measure, Corollary \ref{coro:sync_indep} states that the parameter of the P\'olya-Aeppli asymptotic distribution is $-(m-1)\mathcal R_\mu(m-1)$.  
\end{remark}

\begin{remark}\label{rem:markov_pressure}
In the particular case in which $g^{(i)},i=1,\ldots,m$  only depend on the two first coordinates, $\mu^{(i)}, i=1,\ldots,m$ are Markov chains with matrices $Q^{(i)}(a,b):=g^{(i)}(x)$ for any $x$ such that $x_1=b$ and $x_2=a$. In this case, $P(\sum_{i}\log g^{(i)})$ is $\log \rho$ where $\rho$ is the largest positive eigenvalue of the matrix $Q^\Delta$ defined through $Q^\Delta(a,b)=\prod_{i=1}^mQ^{(i)}(a,b)$.

As an example, consider $m=2$ with $\mathcal A=\{0,1\}$ and take $Q^{(1)}(0,0)=0.2$, $Q^{(1)}(1,1)=0.7$, $Q^{(2)}(0,0)=0.8$, $Q^{(2)}(1,1)=0.9$. Then, we have $Q^\Delta(0,0)=0.16=Q^\Delta(0,1)$, $Q^\Delta(1,0)=0.03$, $Q^\Delta(1,1)=0.63$, and in particular $p=\frac{16}{25}$. 
\end{remark}

Theorem \ref{theo:tempsync} is abstract because it is not stated in terms of the interaction of (possibly distinct) given $g$-measures. The natural question now is  \emph{how to make $g$-measures interact?}
A direct application of the \emph{coupled map lattice approach} used for instance in \cite{faranda2018extreme,haydn2020limiting} (see also Subsection \ref{subsec:uncoupled} below) does not seem to make much sense in the setting of $g$-measures. An observation at this point is that we prefer to use the terminology ``interacting $g$-measures'' instead of ``coupled $g$-measures'', because the second one has a precise definition in stochastic processes, which does not necessarily corresponds to what we want here.

We consider two ways. The first way to make $g$-measures interact is through a coupling of their $g$-functions, coupling in the sense of stochastic processes as we now explain.  Suppose we have $m\ge2$ possibly distinct $g$-functions $g^{(1)},\ldots,g^{(m)}$ on $\Sigma$, and use the notation $\vec{x}_k=(x_k(1),\ldots,x_k(m))\in\mathcal A^m$, $k\ge1$, and $\vec{x}=(\vec{x}_1\vec{x}_2\ldots)$. Then, a $g$-function $\hat g$ on $\Sigma^m$ is said to be a coupling $g$-function of the $g^{(k)}$'s if, for any $k=1,\ldots,m$, any $a\in\mathcal A$ and any $(\vec{x}_2\vec{x}_3\ldots)$
\[
\sum_{\vec{x}_1:x_1(k)=a}\hat g(\vec{x})=g^{(k)}(ax_2(k)x_3(k)\ldots).
\] 
The $g$-measure $\hat \mu$ associated to $\hat g$ is then automatically a coupling of the $g$-measures $\mu^{(k)}$ associated to $g^{(k)}$, $k=1,\ldots,m$, in the sense that the $k^{\text{th}}$ marginal of $\hat\mu$ is equally distributed to $\mu^{(k)}$ for any $k$. An example is given in Subsection \ref{sec:maximal}. 

A second way to make $g$-measures interact is, given $m\ge2$ possibly distinct $g$-functions $g^{(1)},\ldots,g^{(m)}$, to construct a $g$-function on the product space $\Sigma^m$ parametrized by a tuning parameter $\gamma\in[0,1]$ indicating the strength of the interaction, that is, having the property that, when $\gamma=0$, $\hat g=\prod_{i=1}^mg^{(i)}$, which corresponds to the non-interacting case. The resulting $\hat g$ needs not to be a coupling of the $g^{(k)}$'s in the stochastic process meaning. We give a simple example in Subsection \ref{sec:coupled_coupling}.

In order to simplify the presentation, we will restrict ourselves to the case where the $g^{(i)}$'s depend only on the two first coordinates and will construct   $\hat g$'s having the same property. This means that instead of specifying $g$-functions $g$, we will specify matrices $Q$. Moreover, we assume that the $Q^{(i)}$'s and $\hat Q$ have only strictly positive entry, ensuring that we are in force of all the conditions of Theorem \ref{theo:tempsync},  item (2). 

Observe finally that, also according to Theorem \ref{theo:tempsync}, we only have to define the coupling  $\hat g$ on $\Delta$, which means that we only have to define $Q^\Delta$.

\subsubsection{Example 1: the maximal coupling}\label{sec:maximal}
The \emph{maximal coupling} is a classical coupling in the theory of stochastic processes, but is less known in dynamical systems. We refer to \cite{bressaud/fernandez/galves/1999a} for a complete definition, here we only define the coupling on $\Delta$, which is sufficient for our purposes. We start with matrices $Q^{(i)},i=1,\ldots,m$, then the maximal coupling is defined on the diagonal as 
\begin{align*}
\hat Q_{\max}((a,\ldots,a),(b(1),\ldots,b(m))):=\inf\{Q^{(i)}(a,b(i)),i=1,\ldots,m\}.
\end{align*}

So, taking $m=2$ and coming back to the matrices used in Remark \ref{rem:markov_pressure},  we obtain $Q_{\max}^\Delta(0,0)=0.2,Q_{\max}^\Delta(0,1)=0.2, Q_{\max}^\Delta(1,0)=0.1, Q_{\max}^\Delta(1,1)=0.7$ and therefore, according to Theorem \ref{theo:tempsync} and Remark \ref{rem:markov_pressure0} we have asymptotically Polya-Aeppli distributed synchronzations with parameter $p_{\max}=\rho_{\max}$,  the largest eigenvalue of $Q_{\max}^\Delta$. This gives $p_{\max}=\left(\frac{9+\sqrt{33}}{20}\right)$, strictly larger than the value $p=16/25$ obtained in the uncoupled case considered in Remark \ref{rem:markov_pressure}. 

The terminology \emph{maximal} comes from the fact that $\hat Q_{\max}$ puts as much probability as possible on the diagonal, that is, as much probability of agreement as possible in one step, still keeping the marginals equal to $Q^{(i)},i=1,\ldots,m$. So it is natural that the synchronizations last longer than in the uncoupled case, and this is what $p_{\max}>p$ means. 

\subsubsection{Example 2: parametrized coupling} \label{sec:coupled_coupling}
Let $Q^{(1)}$ and $Q^{(2)}$ be two stochastic matrices on $\mathcal A=\{0,1\}$ and define transition probabilities $q^{(1)},q^{(2)}:\mathcal A^2\times \mathcal A\rightarrow[0,1]$ through
\begin{align*}
q^{(1)}((a(1),a(2)),1)&=(1-\gamma)Q^{(1)}(a(1),1)+\gamma a(2)\\
q^{(2)}((a(1),a(2)),1)&=(1-\gamma)Q^{(2)}(a(2),1)+\gamma a(1).
\end{align*}
Naturally, we put $q^{(i)}((a(1),a(2)),0)=1-q^{(i)}((a(1),a(2)),1)$. 

Now, define 
\[
\hat Q_{\gamma}((a(1),a(2)),(b(1),b(2)))=q^{(1)}((a(1),a(2)),b(1))q^{(2)}((a(1),a(2)),b(2)).
\]
Observe that $\hat Q_{\gamma}$ is indeed a stochastic matrix on $\mathcal A^2$, and that when $\gamma=0$, we get $\hat Q_{0}((a(1),a(2)),(b(1),b(2)))=Q^{(1)}(a(1),b(1))Q^{(2)}(a(2), b(2))$ as we wanted. 

As an example, consider the matrices used in Remark \ref{rem:markov_pressure}. According to Theorem \ref{theo:tempsync} and Remark \ref{rem:markov_pressure}, we  have asymptotically Polya-Aeppli with parameter $p_{\gamma}=e^{P(\log Q_{\gamma}^\Delta)}$ which equals the largest eigenvalue of $Q_{\gamma}^\Delta$. This yields
\[
p_{\gamma}=\frac{1}{200}\left(79+2\gamma+19\gamma^{2}+\sqrt{2401+7996\gamma+3006\gamma^{2}-7604\gamma^{3}+4201\gamma^{4}}\right). 
\]
So we notice that the interacting parameter $\gamma$ modifies in a non-trivial way the parameter of the asymptotic distribution. Actually, in the present case, synchronisation increases as the parameter $\gamma$ increases. (Observe that when $\gamma=0$, we retrieve $16/25$, as in the non-interacting case, which is natural.)

The original inspiration here is as a toy model for two interacting neurons, in which the value $1$ means that the neuron is spiking, and $0$ means it is resting. We assume that each neuron, when they don't interact ($\gamma=0$), has a spiking dynamic given by the $Q^{(i)}$, $i=1,2$. When they interact ($\gamma>0$), the probability that a neuron spikes will depend not only on whether or not it just spiked, but also on whether or not the other neuron just spiked, this is what $q^{(i)}$ models. More precisely, it models the effect of excitatory neurons, since the probability of spiking for one neuron increases when the other neuron has spiked: $q^{(1)}((a(1),1),1)>Q^{(1)}(a(1),1)$. Next, given the past, we assume that the probability of spiking for each neuron is independent, this is why $\hat Q_{\gamma}$ is the product $q^{(1)}q^{(2)}$. 

Naturally, this model is very simple and does not represent correctly the complexity of a system of interacting neurons, however, it retains some features of a recent model introduced by \cite{galves2013infinite}. Their model is not markovian, and include several physiological considerations, and it is a natural and interesting problem to analyse the synchonisation properties in their setting, using Theorem \ref{theo:tempsync}.


\section{Discussion on synchronisation}\label{sec:geovscyl}

 The present paper is mainly concerned with the symbolic setting of deterministic dynamical systems. In the present section we make a small digression to discuss the difference, with regard to asymptotic synchronisation, between three situations: (1) the \emph{geometric} approach of deterministic dynamical systems, (2) the \emph{symbolic} approach of deterministic dynamical systems, and (3) Markov chain, a particular case of random dynamical systems. 
 
 We start this section by making a rapid overview of what is known in the geometric setting. The first application of recurrence  type argument, like those used in this paper, to synchronisation was given in section 4 of the article \cite{keller2009rare}, where the authors explicitly computed a first order formula for the leading eigenvalue of the perturbed transfer operator, the perturbation being the small neighborhood around the diagonal. It was successively  shown by \cite{keller2012rare} that  such a perturbative formula was intimately related to the extremal index. This spectral approach to extreme value theory was developed in \cite{faranda2018extreme}, which  showed that the probability of the appearance of synchronization in chaotic coupled map lattices was related to the distribution of the maximum of a certain observable evaluated along almost all orbits. The statistics of the number of visits was proven in \cite{haydn2020limiting}, with a technique different from  the spectral approach: we recall it in the next subsection and then, by considering a very simple example in Section \ref{sec:geovssymbo} we will exhibit clearly how different can be this approach from the symbolic approach developed  in Section \ref{sec:temporal_sync}. We remind that   an alternate probabilistic approach
in a coupled maps setting is proposed in \cite{carney2021extremes}. We conclude with Section \ref{sec:markov}, considering the case of continuous state Markov chains, a third situation, in which yet another behaviour is displayed.

\subsection{Synchronisation of (un)coupled map lattices: geometric approach} \label{subsec:uncoupled}

We shall consider coupled map lattices over uniformly expanding interval maps. 
Let $T$ be a piecewise continuous map on the unit interval $I=[0,1]$ which is uniformly expanding,
i.e.\ satisfies $\inf|DT|>1$. We also assume that $T^{-1}$ has only finitely many branches.
Then we define the coupled map 
 $\hat{T}$ on $\Omega=I^m$,   for some integer $m\ge2$ by
\begin{equation}\label{TMn}
\T(\vec{x})_i=(1-\gamma)T(x_i)+\gamma \sum_{j=1}^m M_{i,j}T(x_j)\qquad \forall\, i=1, 2,\dots,m,
\end{equation}
for $\vec{x}\in \Omega$, where $M$ is an $m\times m$ stochastic matrix and  $\gamma\in[0,1]$ is a
coupling constant. The uncoupled case corresponds to $\gamma=0$ in which case $\T$ is the product
of $m$ copies of $T$. For $\nu>0$ small, we put
\begin{equation}\label{eq:strip_def}
S_{\nu}:=\{\vec{x}\in [0,1]^m: |x_i-x_j|\le \nu \:,\forall i,j\}
\end{equation}
for a tubular neighbourhood of the diagonal $\Gamma$ (see Figure \ref{fig1} for a picture with $m=2$). Then we define as before
$\hat\alpha_{k+1} =\lim_{K\to\infty}\lim_{\nu\to0}\hat\alpha_{k+1}(K,S_\nu)$ for the 
parameters of the limiting compound Poisson distribution which describes the sychronisation
effect in the neighbourhood of the diagonal $\Gamma$.

It has been previously shown by {\cite{haydn2020limiting}} that if $T$ is a piece-wise  uniformly expanding map of the unit interval 
with finitely many branches satisfying a mild geometric condition along the diagonal and if $\mu$ is an
 equilibrium state  for a sufficiently regular potential function on $\Omega$ then
 the compound Poisson parameters are given by
\begin{equation}\label{eq:strip}
 \hat{\alpha}_{k+1}
  =\frac1{(1-\g)^{k(m-1)} \int_{I}h((x)^m) \,dx}\int_{I} \frac{h((x)^m)}{|DT^k(x)|^{m-1}}\,dx
 \end{equation}
 where $h:\Omega\to\mathbb{R}^+$ is the density function of $\mu$ and 
 $(x)^m$ denotes the set of points on the diagonal. 
 
 Obtaining explicit results is still a complicated problem for coupled map lattices. So we now turn our attention to the case of \emph{uncoupled} map lattices, that is, to the case where $\gamma=0$. Such a situation was first investigated by \cite{coelho1994asymptotic}. They proved that, for an absolutely continuous measure of a piecewise 
expanding and smooth map of the circle, the asymptotic distribution of synchronisation is compound Poisson, 
and identified the limiting parameters $\hat\alpha_k,k\ge1$.

So let us see how \eqref{eq:strip} looks like in the uncoupled case (we consider here the case $m=2$ to simplify) in the setting of interval transformations.  Consider  a 
partition $\mathcal{A}=\{I_1,\ldots,I_M\}$ of $I:=[0,1)$ and a piecewise linear Markov transformation $T$ 
which is continuous, monotone and uniformly expanding on each of the sub intervals $I_i$,
that is  $\inf_{I_i}|T'_i|>1$ for any $i=1,\ldots,M$, where $T_i=T$ on $I_i$. Define the stochastic
$M\times M$ matrix $Q$ by
\[
Q_{i,j}=\begin{cases}0&\mbox{ if $I_j\cap T(I_i)=\varnothing$}\\
\frac{1}{|T'_i|}&\mbox{ if $I_j\subset T(I_i)$}.\end{cases}
\]
We know in this case that the invariant density $h$ which satisfies $\mathcal{L}h=h$, where 
$\mathcal{L}$ is the transfer operator, is piecewise constant. Thus put   $h_i=h(x)$ for  $x\in I_i$ 
and consider the row vector $\vec{h}=(h_1,\ldots, h_M)$. Then  $hQ=h$.

According to~\cite{faranda2018extreme} and \cite{haydn2020limiting}, we then have to compute 
\[
\hat\alpha_{k+1}=\frac{\int_I\frac{h^2(x)}{|DT^k(x)|}\,dx}{\int_I h^2(x)\,dx}.
\]
For any $(a_1,\ldots,a_k)\in\{1,\ldots,M\}^k$, we  use the notation $x\in I_{(a_1,\ldots,a_k)}$ for $T^{i-1}x\in I_{a_i},\,i=1,\ldots,k$ and let 
\[
\mathcal{A}^k:=\{(a_1,\ldots,a_k)\in\{1,\ldots,M\}: Q_{a_i,a_{i+1}}>0,i=1,\ldots,k-1\}.
\] 
If $(a_1,\ldots,a_k)\in\mathcal{A}^k$ then, using the chain rule,
\[
|DT^k(x)|=|\prod_{i=1}^{k}DT(T^{i-1}x)|=\prod_{i=1}^{k}|DT_{a_i}|\,\,,\,\,\,\,\forall x\in I_{(a_1,\ldots,a_k)}
\]
while on the other hand
\[
\lambda(I_{(a_1,\ldots,a_k)})=\frac{1}{\prod_{i=1}^{k}|DT_{a_i}|}.
\]
We can now compute 
\begin{align}
\int\frac{h^2(x)}{|DT^k(x)|}\,dx&=\sum_{(a_1,\ldots,a_k)\in\mathcal{A}^k}h_{a_1}^2\frac{\lambda(I_{(a_1,\ldots,a_k)})}{\prod_{i=1}^{k}|DT_{a_i}|}\label{eq:diagonal_formula}
\end{align}
and
\begin{align*}
\int_Ih^2(x)\,dx&=\sum_ih_i^2\lambda(I_i).
\end{align*}
 {For instance, consider the case of $T(x)=3x\,\text{mod}\,1$. Then we have  $h_i=1$ and $|DT_i|=3$ for $i=1,2,3$,  and thus $\int\frac{h^2(x)}{|DT^k(x)|}dx=3^{-k}$ while $\int_Ih^2(x)dx=1$. We therefore obtain $\hat\alpha_{k+1}=3^{-k}$, which means that the random variable $W$ has P\'olya-Aeppli distribution  with parameter $t(1-1/3)$ (see Subsection \ref{rem:compPois}). A natural conjecture (having also in view Theorem \ref{theo:tempsync}) would be that this 
 is the case for any uncoupled (sufficiently mixing) dynamical systems. However, as we show below, this is not necessarily the case.}

\subsection{Geometric \emph{vs.} symbolic: an explicit example}\label{sec:geovssymbo}

As is probably clear by now, we call geometric approach when we measure the synchronisation through   visits to $S_\nu,\nu>0$ (see \eqref{eq:strip_def}), thin strips around the diagonal, as pictured in Figure \ref{fig1}. On the other hand,  the symbolic approach, when we measure synchronisation through visits to cylinder sets covering the diagonal, is pictured in Figure \ref{fig2}. 
\begin{figure}[!h]
           \begin{floatrow}
             \ffigbox{\includegraphics[scale = 0.315]{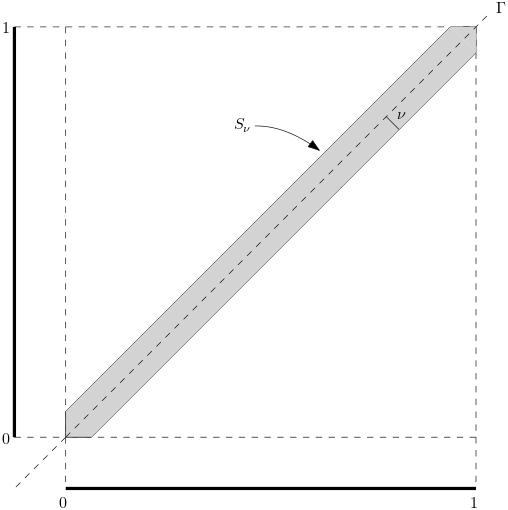}}{\caption{A strip $S_\nu$ around the diagonal $\Gamma$ in $[0,1)^2$.}\label{fig1}}
             \ffigbox{\includegraphics[scale = 0.31]{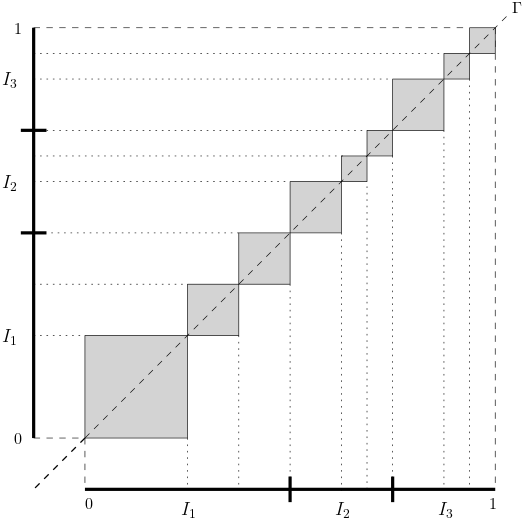}}{\caption{The set $S_2=\bigcup_{A\in\mathcal{A}^2}(A\times A)$ where the partition $\mathcal A=\{[0,\frac{1}{2}),[\frac12,\frac34),[\frac34,1)\}$.}\label{fig2}}
           \end{floatrow}
        \end{figure}
 
Here we show that both approaches yield  different classes of distributions in general (although both compound Poisson), by mean of a simple example. 
Consider the following map $T$ on the unit interval $I$:
\begin{equation}\label{eq:interval_tranformation}
T(x)=\begin{cases}3x & \mbox{ if $ x\in I_1:=[0,1/3)$}\\
5/3-2x & \mbox{ if $ x\in I_2:=[1/3,2/3)$}\\
-2+3x & \mbox{ if $ x \in I_3:=[2/3,1)$}.\\
\end{cases}
\end{equation}
It is a piecewise linear Markov transformation with
$$
Q=\left(\begin{array}{ccc}\frac13&\frac13&\frac13\\
0&\frac12&\frac12\\
\frac13&\frac13&\frac13
\end{array}\right).
$$

\subsubsection{Geometric approach} We know \citep{haydn2020limiting,faranda2018extreme} that   the asymptotic distribution of synchronisation is Poisson compound. We now calculate explicitly the parameters $\hat{\alpha}_{k+1},k\ge1$ using \eqref{eq:diagonal_formula} which, in matricial form, gives
\[
\int\frac{h^2(x)}{|DT^k(x)|}\,dx=\sum_{(a_1,\ldots,a_k)\in\mathcal{A}^k}h^2_{a_1}\frac{\lambda(I_{(a_1,\ldots,a_k)})}{\prod_{i=1}^{k}|DT_{a_i}|}=\sum_{(a_1,\ldots,a_k)}h_{a_1}^2\prod_{i=1}^{k}Q^2_{a_i,a_{i+1}}=\text{Trace}(\nu^{h}\hat Q^{k})
\]
where $\nu^h=(h^2_1\ldots h^2_n)$ and $\hat Q$ is the matrix with entries $\hat Q_{i,j}=Q^2_{i,j}$. The corresponding  piecewise constant density with respect to Lebesgue is the row vector $\vec{h}=(3/5\,\,\,6/5\,\,\,6/5)$, and we then obtain
$\nu^h=(9/25\,\,\,36/25\,\,\,36/25)$. With this example we get (using Mathematica online) that $\text{Trace}(\nu^{h}\hat Q^{k})$ equals
\[
\frac{2^{-3 k - 1} 3^{2 - 2 k} \left[(3 \sqrt{145} - 23) (17 - \sqrt{145})^k + (23 + 3 \sqrt{145}) (17 + \sqrt{145})^k\right]}{25 \sqrt{145}}
\]
and 
\[
\int_Ih^2(x)\,dx=\sum_ih_i^2\lambda(I_i)=27/25.
\]
This yields
\begin{align*}
\hat\alpha_{k+1}&=\frac{2^{-3 k - 1} 3^{2 - 2 k} \left[(3 \sqrt{145} - 23) (17 - \sqrt{145})^k + (23 + 3 \sqrt{145}) (17 + \sqrt{145})^k\right]}{27 \sqrt{145}}
\end{align*}
which does not correspond to a geometric distribution and we do not have P\'olya-Aeppli asymptotic distribution of synchronisation in a strip around the diagonal.

\subsubsection{Symbolic approach} We can use Theorem \ref{theo:tempsync} doing $g(x)=:Q_{i,j}$ for any $x$ such that $x_1=i,x_2=j$ (\emph{i.e.} $g(x)$ depends only on the first two coordinates $x_1, x_2$). According to Remark \ref{rem:markov_pressure}, we conclude that  synchronisation distribution  for cylinder neighbourhoods of the diagonal  converges to P\'olya-Aeppli with the parameter $t(1-\rho)$, where $\rho$, the largest positive eigenvalue of $\hat Q$,  equals $\frac1{72}(17+\sqrt{145})$ for this specific example.

\subsection{Synchronisation of Markov chains}\label{sec:markov}
Let us in this section consider the quintessential random dynamical systems which are Markov chains.
In fact a {\em random transformation} can in a simple way lead to a Markov chain in the following way
(see for instance \cite{bahsoun2014pseudo}).
Take $\{\omega_k\}_{k\in \mathbb{N}_0}$ a sequence of i.i.d.\
 random variables with values
in some $\Omega$ that carries the probability measure $\theta$.
 We associate to each $\omega\in
\Omega$ a map $T_{\omega}$
and the iteration of the unperturbed map $T^n(x)$,
 will be replaced by the composition of
random maps $T_{\omega_{n-1}}\circ\cdots\circ T_{\omega_0}.$
These random transformations generate a stationary
 Markov chain $\{Z_n\}_{n\ge 0}$ with transition probabilities,
for any $n\ge 1:$
$$
\mathbb{P}(Z_{n+1}\in A|Z_n=x)=
\int_{\Omega}{\bf 1}_A(T_{\omega}(x))\,d\theta(\omega),
$$
where $A$ is a measurable set in some $M$ and $T_\omega$ is a map from $M$ to $M$.

We therefore consider such a Markov chain  $\{X_n\}_{n\ge 0}$,  stationary, with  continuous
state space $I\subset \mathbb R$ and transition probabilities
$p(x,A)=\mathbb{P}(X_{n+1}\in A|X_n=x)$ for measurable sets $A$.
If the transition kernel has a density $p(x,y)$, that is if $p(x,dy)=p(x,y)\,dy$, then
$$
\mathbb{P}(X_{n+1}\in A|X_n=x)=\int_A p(x,y)\,dy.
$$
The invariant measure $\mathbb{P}$ is then given by the transition probabilities $p$ and
 an initial probability measure $\rho$. That is  $\mathbb{P}(X_0\in A)=\int_A\,d\rho(y)$
 and 
 \[
 d\mathbb{P}(x_0,x_1,x_2,\dots,x_n)=d\rho(x_0)p(x_0,dx_1)p(x_1,dx_2)\cdots p(x_{n-1},dx_n).
 \]
The probability measure $\rho$ on $I$ is invariant under the map $T$ which is given by
$T\mu(A)= \int_I p(x,A)\,d\mu(x)$, for all $A$ measurable, and is the annealed invariant measure on $I$.
 The Markov chain satisfies the {\em Doeblin condition} if there exists a 
probabiltiy measure $\nu$ 
and an $\eta\in(0,1)$ so that $p(x,A)\ge (1-\eta)\nu(A)$ for all measurable $A$. If $p$ satisfies
 this condition then in the total variation norm $\|p^n(x,\cdot)-\rho(\cdot)\|_{TV}\le2\eta^n$ uniformly in $x$.
 
We can now  associate to $\{X_n\}_{n\ge 0}$ another independent copy $\{Y_n\}_{n\ge 0}$ and ask the distribution of the first synchronisation 
time of the two chains. The Markov chain $(X_n,Y_n)\in I\times I$, $n\in\mathbb{N}_0$, has 
transition probabilities $p_2((x,y),(A,B))=p(x,A)p(y,B)$, $A,B$ measurable.
Let us denote by $\hat{T}$ the product map on $I\times I$ and by $ \mathbb{P}_2=\mathbb{P}\otimes\mathbb{P}$ its
invariant probability measure.
By using the procedure of section \ref{subsec:uncoupled}, we are led
  to consider the direct product of these two chains and  look at the couples of points   which stay close to  each other up to  time $n-1$.
Let us therefore set $\Delta_{\delta}:=\{(x,y)\in I_2=I\times I, |x-y|\le \delta\}$ for a neighbourhood of the 
diagonal in $I\times I$. 

Specifically, we will show that for Markov chains whose densities are bounded above and 
away from $0$, the limiting   distribution (as $\delta\rightarrow0$) of return times to $\Delta_\delta$  is Poissonian, which means that orbits don't 
cluster over time and that there is no sychronisation effect.

 In order to get a limiting compound Poisson distribution as $\delta\to0$ 
we want to use Theorem~3 
 of \cite{haydn2020limiting}. For that purpose 
 let us put $\mathfrak{X}_i=\ind_{\Delta_\delta}\circ \hat{T}^i$, $i=0,1,2,\dots$, and 
 $\mathfrak{W}_a^b=\sum_{i=a}^b\mathfrak{X}_i$. For simplicity we put 
$\mathfrak{W}=\mathfrak{W}_0^N$,
 where $N=t/\mathbb{P}_2(\Delta_\delta)$ (take integer part) where $t>0$ is a parameter. 
 We cut the time interval $N$ into
 blocks of length $2K+1$ for some large $K$ ($K<\!\!<N$) and put
  $\mathfrak{Z}=\sum_{i=0}^{2K}\mathfrak{X}_i$. If we put $N'=N/(2K+1)$ (assuming it being an
  integer) then $\mathfrak{W}=\sum_{n=0}^{N'-1}\mathfrak{Z}\circ \hat{T}^{n(2K+1)}$.
  We now choose a gap $\gamma<\!\!<N'$ ($\gamma\ge2$) and want to estimate the quantities
\begin{align*}
\mathcal{R}_1 &= \sup_{\substack {0< \gamma<M\le N'\\0 <q<N'-\gamma-1/2}}
 \left|\sum_{u=1}^{q-1}\!\left(\mathbb{P}_2\!\left(\mathfrak{Z}=u \land \mathfrak{W}_{\gamma(2K+1)}^{M(2K+1)}=q-u\right)\right.\right.\\
&\quad\quad\quad\quad\quad\quad\quad\quad\quad\quad\quad\quad\left.\left.-\mathbb{P}_2(\mathfrak{Z}=u)\mathbb{P}_2\!\left(\mathfrak{W}_{\gamma(2K+1)}^{M(2K+1)}=q-u\right)\right)\right|\\
\mathcal{R}_2 &= \sum_{j=1}^\gamma \mathbb{P}_2(\mathfrak{Z}\ge1 \land \mathfrak{Z}\circ \hat{T}^{(2K+1)j}\ge1).
\end{align*}
  If we denote by $\tilde\nu$ the compound binomial distribution measure where the binomial part
 has values $\mathfrak{p}=\mathbb{P}_2(\mathfrak{Z}\ge1)$ and $N'=N/(2K+1)$ and the compound part has
 probabilities  $\lambda_\ell(K,\delta)=\mathbb{P}_2(\mathfrak{Z}=\ell)/\mathfrak{p}$,
then, by Theorem~3 of \cite{haydn2020limiting},  there exists a constant $c_1$, independent of $K$ and $\gamma$,
such that
$$
|\mathbb{P}_2(\mathfrak{W}=k) - \tilde\nu(\{k\})|
 \leq c_12(N'(\mathcal{R}_1 + \mathcal{R}_2) + \gamma\mathbb{P}_2(\mathfrak{X}_0=1)).
$$
(Please note that Theorem~3 in the paper from 2020 which we use here contains a typo
in the lower summation limit of $j$ in the expression for $\mathcal{R}_2$: As we put it here the lower
summation limit must be $j=1$ and not $j=2$ as printed there.)
The proof of the following lemma is given at the end of the present section.

\begin{lemma}\label{lem:markov}
For any $K'<K$ there exists a constant $c_2$ so that
$$
N'(\mathcal{R}_1+\mathcal{R}_2)\le c_2K\gamma\mathbb{P}_2(\Delta_\delta)+\eta^{K'}+\frac{K'}K.
$$
\end{lemma}

If we put $\gamma=\mathbb{P}_2(\Delta_\delta)^{-\beta}$ for some $\beta\in(0,1)$ then
$$
\mathbb{P}_2(\mathfrak{W}=k)\longrightarrow\nu^K(\{k\})+\mathcal{O}(\eta^{K'}+K'/K)
$$
as $\delta\to0$, where $\nu^K$ is the compound Poisson distribution with parameters
$t\lambda_\ell(K)$ with
$$
\lambda_\ell(K)=\lim_{\delta\to0}\frac{\mathbb{P}_2(\mathfrak{Z}\ge1)}{\mathbb{P}_2(\Delta_\delta)}
$$
assuming the limits exist. Now put e.g.\ $K'=\sqrt{K}$ and let $K$ go to infinity. 
Then $\nu^K$ converges to a compound Poisson distribution $\nu$ with parameters
$t\lambda_\ell$, where $\lambda_\ell=\lim_{K\to\infty}\lambda_\ell(K)$ assuming the 
limits exist. Thus 
$$
\mathbb{P}_2(\mathfrak{W}=k)\longrightarrow\nu(\{k\}).
$$

What we showed above is that, if the transition probabilities are 
given by a density $p(x,y)$ satisfying the Doeblin condition (that is, if we assume that $p$ is also bounded away from $0$), then  $\mathfrak{W}$ converges in distribution to a compound Poisson
distribution $\nu$. In order to show that $\nu$ is in fact a straight Poisson distribution 
we want to show that $\hat\alpha_2=0$, which implies by monotonicity that $\hat\alpha_\ell=0$ for all $\ell\ge2$. To that end, we shall further assume that the transition probabilities $p(x,y)$ are 
 bounded above by some constant $\mathcal{K}$.
If $p(x,y)\le\mathcal{K}$ for all $x,y$, we also get $p^n(x,y)\le\mathcal{K}$ for all $x,y$
and thus
\begin{eqnarray*}
\mathbb{P}_2(\Delta_\delta\cap \hat{T}^{-n}\Delta_\delta)
&=&\int \ind_{\Delta_\delta}(x_0,y_0)\ind_{\Delta_\delta}(x_n,y_n)
\,d\rho(x_0)\,d\rho(y_0)p^n(x_0,x_n)p^n(y_0,y_n)\,dx_n\,dy_n\\
&\le&\mathcal{K}\mathbb{P}_2(\Delta_\delta)\mathfrak{m}(\Delta_\delta),
\end{eqnarray*}
where  $\mathfrak{m}$ is the Lebesgue measure on $I\times I$, .
 Thus 
$$
\hat\alpha_2(K,\delta)
=\frac{\mathbb{P}_2\!\left(\Delta_\delta\cap\bigcup_{n=1}^{2K}\hat{T}^{-n}\Delta_\delta\right)}{\mathbb{P}_2(\Delta_\delta)}
\le 2K\mathcal{K}\mathfrak{m}(\Delta_\delta)\longrightarrow0
$$
as $\delta\to 0$ (since $\mathfrak{m}(\Delta_\delta)\le2\delta$) for all $K$.
Thus $\hat\alpha_\ell=0$  for all $\ell\ge2$ which implies that $\nu$ is Poisson with parameter $t$.

Let us remark that the double limit first $\delta\to0$ and then $K\to\infty$ can be synchronised by
going along a sequence $K_\delta\to\infty$ as $\delta$ approaches $0$ in such a way that
$K_\delta\mathbb{P}_2(\Delta_\delta)\to0$ as $\delta\to0$. This was shown  in \cite[Proposition~6.2]{yang2021rare}.

%

We conclude the section with the proof of the lemma. 
\begin{proof}[Proof of Lemma \ref{lem:markov}] In order to estimate $\mathcal{R}_1$ and the terms of 
$\mathcal{R}_2$ for $j\ge2$ 
we use the Doeblin condition and the consequential exponential convergence to the initial 
distribution to obtain:
\begin{eqnarray*}
\mathbb{P}_2(\mathfrak{Z}\ge1,\mathfrak{Z}\circ \hat{T}^{j(2K+1)}\ge1)
&=&\int \ind_{\mathfrak{Z}\ge1}(z_0,\dots,z_{2K})\ind_{\mathfrak{Z}\ge1}(z'_0,\dots,z'_{2K})\\
&&\hspace{1cm}d\rho_2(z_0)p_2(z_0,dz_1)\cdots p_2(z_{2K-1},dz_{2K})\\
&&\hspace{1cm}p_2^{(j-1)(2K+1)}(z_{2K},dz'_0)p_2(z'_0,dz'_1)\cdots p_2(z'_{2K-1},dz'_{2K})\\
&\le&\mathbb{P}_2(\mathfrak{Z}\ge1)\!\left(\mathbb{P}_2(\mathfrak{Z}\ge1)+2\eta^{(j-1)(2K+1)}\right),
\end{eqnarray*}
where $z_i=(x_i,y_i)$ and  $\rho_2=\rho\times \rho$.
Consequently
$$
\mathcal{R}_1
\le \sum_{j=\gamma}^\infty\mathbb{P}_2(\mathfrak{Z}\ge1)\eta^{(j-1)(2K+1)}
\le \mathbb{P}_2(\mathfrak{Z}\ge1)\eta^{(\gamma-1)(2K+1)}.
$$
For the estimate of $\mathcal{R}_2$ we consider the case $j=1$ separately, choose $K'<K$ 
and put $\mathfrak{Z}'=\sum_{i=2K-K'}^{2K}\mathfrak{X}_i$, 
$\mathfrak{Z}''=\mathfrak{Z}-\mathfrak{Z}'=\sum_{i=0}^{2K-K'-1}\mathfrak{X}_i$.
Then, similarly as above, we obtain
$$
\mathbb{P}_2(\mathfrak{Z}''\ge1,\mathfrak{Z}\circ \hat{T}^{2K+1}\ge1)
\le \mathbb{P}_2(\mathfrak{Z}''\ge1)\!\left(\mathbb{P}_2(\mathfrak{Z}\ge1)+2\eta^{K'}\right).
$$
Since 
$$
\mathbb{P}_2(\mathfrak{Z}\ge1,\mathfrak{Z}\circ \hat{T}^{2K+1}\ge1)
\le \mathbb{P}_2(\mathfrak{Z}''\ge1,\mathfrak{Z}\circ \hat{T}^{2K+1}\ge1)+\mathbb{P}_2(\mathfrak{Z}'\ge1)
$$
we obtain 
\begin{eqnarray*}
\mathcal{R}_2
&\le& \mathbb{P}_2(\mathfrak{Z}''\ge1)\!\left(\mathbb{P}_2(\mathfrak{Z}\ge1)+2\eta^{K'}\right)
+\mathbb{P}_2(\mathfrak{Z}'\ge1)
+2\sum_{j=2}^\gamma\mathbb{P}_2(\mathfrak{Z}\ge1)\eta^{(j-1)(2K+1)}.
\end{eqnarray*}
For the final estimate we use that $N'=\frac{N}{2K+1}\le\frac{N}K=\frac{t}{K\mathbb{P}_2(\Delta_\delta)}$
and $\mathbb{P}_2(\mathfrak{Z}\ge1)\le(2K+1)\mathbb{P}_2(\Delta_\delta)$, 

$\mathbb{P}_2(\mathfrak{Z}'\ge1)\le K'\mathbb{P}_2(\Delta_\delta)$.
Then as $K'<K$ 
\begin{eqnarray*}
N'(\mathcal{R}_1+\mathcal{R}_2)
&\lesssim& \frac{t}{K\mathbb{P}_2(\Delta_\delta)}\!\left(3K^2\mathbb{P}_2(\Delta_\delta)^2
+K\mathbb{P}_2(\Delta_\delta)\eta^{(\gamma-1)K}
+2K\mathbb{P}_2(\Delta_\delta)\eta^{K'}+K'\mathbb{P}_2(\Delta_\delta)\right)\\
&\lesssim&K\mathbb{P}_2(\Delta_\delta)
+\eta^{(\gamma-1)K}
+\eta^{K'}+\frac{K'}K
\end{eqnarray*}
which implies the statement as $K'<(\gamma-1)K$.
\end{proof}

\section{Proofs of general theorems}  \label{sec:proofs}

Before we come to the proofs of Theorems \ref{theorem.compound.poisson}, \ref{corollary.compound.poisson_psi} and \ref{main.theorem} we start this section with an important subsection which describes rapidly the  classical Stein-Chen method, used to prove Theorem \ref{theorem.compound.poisson}. 


\subsection{Stein-Chen method}
  We will use the Stein-Chen method as  described in~\cite{roos1994stein} to estimate
   how close a given 
 probability measure $\nu$ is to a compound Poisson distribution $\tilde\nu$
 for parameters $t\lambda_\ell$, $\ell=1,2,\dots$, which satisfy 
 $\sum_\ell\lambda_\ell<\infty$.
 On the  space $\mathcal{F}=\{f: \mathbb{N}_0\to\mathbb{R}\}$
  of functions on the non-negative integers, the
 {\em Stein operator} $\mathscr{S}:\mathcal{F}\to\mathcal{F}$ is defined by
$$
\mathscr{S}g(k)=kg(k)-\sum_{\ell=1}^\infty t\ell\lambda_\ell g(k+\ell).
$$
For a given set $E\subset\mathbb{N}_0$ one wants to find $f$ so that 
$$
\mathscr{S}f=\ind_E-\tilde\nu(E)
$$
where $\tilde\nu$ is the compound Poisson distribution with parameters $t\lambda_\ell$.
This last identity is the Stein equation. Proposition~1 of \cite{barbour1992compound}  states that for given 
$E\subset\mathbb{N}_0$ the solution $f$ satisfies $f(k)\lesssim\frac1k$.

Indeed if $\int \mathscr{S}f\,d\tilde\nu=0$
 for all bounded functions $f$ on $\mathbb{N}$, then
\begin{eqnarray*}
0&=&\sum_k\mathscr{S}f(k)\tilde\nu(\{k\})\\
&=&\sum_k\left(kf(k)-\sum_{\ell=1}^\infty t\ell\tilde\lambda_\ell f(k+\ell)\right)\tilde\nu(\{k\})\\
&=&\sum_k f(k)\!\left(k\tilde\nu(\{k\})-\sum_{\ell=1}^k t\ell\tilde\lambda_\ell\tilde\nu(\{k-\ell\})\right)
\end{eqnarray*}
implies
$$
k\tilde\nu(\{k\})=\sum_{\ell=1}^k t\ell\tilde\lambda_\ell\tilde\nu(\{k-\ell\})
$$
 for every $k$.  From this we conclude that $\tilde\nu$ has the generating function 
$$
\varphi_{\tilde\nu}(z)=\exp\sum_\ell t\tilde\lambda_\ell(e^{z\ell}-1)
$$
which equals $\exp\int_0^\infty(e^{zx}-1)\,d\rho(x)=\exp(\varphi_\rho(z)-L)$ where 
$\rho =\sum_\ell t\tilde\lambda_\ell\delta_\ell$, $L=\sum_\ell t\tilde\lambda_\ell$ and
 $\varphi_\rho(z)=\sum_\ell t\tilde\lambda_\ell e^{z\ell}$
is the generating function for the measure $\rho$. 
This implies that $\tilde\nu$ is compound Poisson with parameters $t\tilde\lambda_\ell$.

\subsection{Proof of Theorem~\ref{theorem.compound.poisson}}

We are now ready to prove our first main result. 
\begin{proof}[Proof of Theorem~\ref{theorem.compound.poisson}]
Let $\tilde\nu$ be the compound Poisson distribution for $t{\tilde\lambda_\ell}$, $\ell\in\mathbb{N}$ as defined in the statement of the theorem. For $E\subset{N}_0$ let again $f$ be the solution of the 
Stein equation $\mathscr{S}f=\ind_E-\tilde\nu(E)$. Then for a probability measure
$\nu$ on $\mathbb{N}_0$ we have
\begin{eqnarray*}
\nu(E)-\tilde\nu(E)&=&\int\ind_E\,d\nu-\int\ind_E\,d\tilde\nu\\
&=&\int(\ind_E-\tilde\nu(E))\,d\nu\\
&=&\int\mathscr{S}f\,d\nu\\
&=&\int\left(Wf(W)-\sum_{\ell=1}^\infty t\ell\tilde\lambda_\ell f(W+\ell)\right)\,d\nu(k).
\end{eqnarray*}
For some fixed set $U\subset\Omega$ and $t>0$, recall the definition \eqref{eq:W} of the random variable $W$ and put $\nu(\cdot)=\mu(W\in \cdot)$. We have
\begin{equation}\label{eq:nu}
\mu(W\in E)-\tilde\nu(E)=\int\mathscr{S}f\,d\nu
=\mathbb{E}[Wf(W)]-\sum_{\ell=1}^\infty t\ell\tilde\lambda_\ell\mathbb{E}[f(W+\ell)].
\end{equation}

Now let $K$ and $\Delta$ be (later taken to be large) numbers so that $K<\!\!<\Delta<\!\!<t/\mu(U)$ and 
let us establish the following notation (with the obvious restrictions 
$i-j\ge0$ and $i+j\le t/\mu(U)$):\\
\begin{enumerate}[(i)] 
\item Close range interactions: $Z_i=\sum_{j=-K}^KI_{i+j}$. (The purpose of the double sided sum 
centred at $i$ is to cover both cases, when $T$ is non-invertible as well as the case when $T$ is 
invertible.) Observe that was denoted $Z_i^{(K)}$ in the beginning of the paper but we will omit the upper script to avoid overloaded notation.\\
\item The gap terms
$$
 V^{-}_i=\sum_{j=K+1}^{K+\Delta} I_{i-j},\quad V^{+}_i=\sum_{j=K+1}^{K+\Delta} I_{i+j}
$$
and $V_i=V_i^-+V_i^+$ for the entire gap.\\
\item The principal terms:
$$
Y_i^-=\sum_{j>K+\Delta}I_{i-j},\quad Y_i^+=\sum_{j>K+\Delta}I_{i+j},
$$
and $Y_i=Y_i^-+Y_i^+$ for the entire principle term. We reforce that these summands are restricted by $i-j\ge0$ and $i+j\le t/\mu(U)$, and are therefore not infinite. 
\end{enumerate}

In this way we decomposed $W$ as $W=Z_i+V_i+Y_i$ for every $i=1,\dots,N$.

Now observe that, by translation invariance
for any $i\in(K,N-K)$
\[
t\ell\tilde\lambda_\ell=t\mathbb{E}(\ind_{Z_i=\ell}|I_i=1)=\frac{t}{\mu(U)}\mathbb{E}(I_i\ind_{Z_i=\ell})=\sum_{i=1}^{t/\mu(U)}\mathbb{E}(I_i\ind_{Z_i=\ell}).
\]
Thus  
\[
t\ell{\tilde\lambda_\ell}\mathbb{E}f(W+\ell)=\sum_{i=0}^{t/\mu(U)}\mathbb{E}(I_i\ind_{Z_i=\ell})\mathbb{E}f(W+\ell).
\]
On the other hand we can naturally write  
\[
Wf(W)=\sum_iI_if(W).
\]
With this in hand, coming back to \eqref{eq:nu} we have
\begin{eqnarray*}
\mu(W\in E)-\tilde\nu(E)
&=&\sum_{i=0}^{t/\mu(U)}\left(\mathbb{E}[I_if(W)]-\sum_\ell\mathbb{E}[I_i\ind_{Z_i=\ell}]\mathbb{E}[f(W+\ell)]\right)\\
&=&\sum_{i=0}^{t/\mu(U)}\left(\sum_\ell\mathbb{E}[I_i\ind_{Z_i=\ell}f(W)]-\sum_\ell\mathbb{E}[I_i\ind_{Z_i=\ell}]\mathbb{E}[f(W+\ell)]\right)\\
&=&\sum_{i=0}^{t/\mu(U)}\left(\sum_\ell\mathbb{E}[I_i\ind_{Z_i=\ell}f(Y_i+V_i+\ell)]-\sum_\ell\mathbb{E}[I_i\ind_{Z_i=\ell}]\mathbb{E}[f(W+\ell)]\right).
\end{eqnarray*}
We now split the error term on the right hand side into three parts as follows:
\begin{eqnarray*}
&&\sum_{i=0}^{t/\mu(U)}\sum_\ell\left(\mathbb{E}[I_i\ind_{Z_i=\ell}f(Y_i+V_i+\ell)]-\mathbb{E}[I_i\ind_{Z_i=\ell}f(Y_i+\ell)]\right)\\
&&\hspace{1cm}+\sum_{i=0}^{t/\mu(U)}\sum_\ell\left(\mathbb{E}[I_i\ind_{Z_i=\ell}f(Y_i+\ell)]-\mathbb{E}[I_i\ind_{Z_i=\ell}]\mathbb{E}[f(Y_i+\ell)]\right)\\
&&\hspace{1cm}+\sum_{i=0}^{t/\mu(U)}\sum_\ell\left(\mathbb{E}[I_i\ind_{Z_i=\ell}]\mathbb{E}[f(Y_i+\ell)]-\mathbb{E}[I_i\ind_{Z_i=\ell}]\mathbb{E}[f(W+\ell)]\right)\\
&&=A+B+C.
\end{eqnarray*}
We now proceed to show that each of the three terms can be upper bounded in order to give the bound stated by Theorem \ref{theorem.compound.poisson}.
\begin{enumerate}[(i)]
\item For the first term we write $A=\sum_{i,\ell}A_{i,\ell}$ where
\begin{eqnarray*}
|A_{i,\ell}|&=&|\mathbb{E}[I_i\ind_{Z_i=\ell}(f(Y_i+V_i+\ell))-f(Y_i+\ell)]|\\
&\le&\|f'\|\mathbb{E}[I_i\ind_{Z_i=\ell}V_i].
\end{eqnarray*}
Note that
\begin{eqnarray*}
\sum_{\ell=1}^{2K+1}\mathbb{E}(I_i\ind_{Z_i=\ell}V_i^+)
&=&\mathbb{E}(I_i\ind_{Z_i\ge1}V_i^+)\\
&\le&\mathbb{E}(I_iV_i^+)\\
&\le&\sum_{j=K+1}^{2n}\mu(U^{j/2}\cap T^{-j}U)
+\sum_{j=2n+1}^{K+\Delta}\mu(U\cap T^{-j}U)
\end{eqnarray*}
where we recall that $U^\ell=A_\ell(U)=\bigcup_{A\in\mathcal{A}^\ell, A\cap U\not=\varnothing}A$ is 
the outer $\ell$-approximation of $U$ ($\ell\le n$).
Therefore by the right $\phi$-mixing property (see \eqref{def:phi_left})
\begin{align}
\label{eq:critical_part_mixing}\sum_\ell\mathbb{E}(I_i\ind_{Z_i=\ell}V_i^+)
\le&\mu(U)\left(\sum_{j=K+1}^{2n}[\mu(U^{j/2})+\phi(j/2)]
+\sum_{j=2n+1}^{K+\Delta}[\mu(U)+\phi(j-n)]\right)\\
\nonumber\le&\mu(U)\left(\sum_{j=K/2}^{n}\mu(U^{j})+\Delta\mu(U)+\sum_{j=K/2}^\infty\phi(j)\right).
\end{align}
One can also show that 
$$
\sum_\ell\mathbb{E}(I_i\ind_{Z_i=\ell}V_i^-)
=\sum_\ell\mathbb{E}(I_i\ind_{Z_i=\ell}V_i^+).
$$
Putting the above together we obtain ($\phi^1$ is the tail sum of $\phi$)
\begin{align*}
|A|&\le \sum_{i=0}^{t/\mu(U)}c_1 \|f'\|\mu(U)\left[\Delta\mu(U)+\phi^1(K/2)+\sum_{j=K/2}^n\mu(U^j)\right]\\
&=c_2 \|f'\|t\left[\Delta\mu(U)+\phi^1(K/2)+\sum_{j=K/2}^n\mu(U^j)\right].
\end{align*}
On the other hand, we have $\|f'\|=\mathcal{O}(1)$ since by Theorem~4 in~\cite{barbour1992compound}  
\[
|f(Y_i+V_i+\ell)-f(Y_i+\ell)|\le c_3\frac1{Y_i+\ell}
=\mathcal{O}(1).
\]
So
\begin{equation}\label{eq:A}
|A|\le c_2c_3 t\left[\Delta\mu(U)+\phi^1(K/2)+\sum_{j=K/2}^n\mu(U^j)\right].
\end{equation}

\item We now estimate  $B$. We get
\begin{eqnarray*}
B&=&\sum_{i}\sum_\ell\sum_{a=0}^\infty f(a+\ell)
\left(\mathbb{E}(I_i\ind_{Z_i=\ell}\ind_{Y_i=a})
-\mathbb{E}(I_i\ind_{Z_i=\ell})\mathbb{E}(\ind_{Y_i=a})\right)\\
&=&\sum_{i,\ell}\sum_{a=0}^N f(a+\ell)\sum_{a^-+a^+=a} 
\left(\mathbb{E}(I_i\ind_{Z_i=\ell}\ind_{Y_i^-=a^-}\ind_{Y_i^+=a^+})
-\mathbb{E}(I_i\ind_{Z_i=\ell})\mathbb{E}(\ind_{Y_i^-=a^-}\ind_{Y_i^+=a^+})\right),
\end{eqnarray*}
as $a\le N$.
We want to sort the terms by their sign so that every level we have only two terms to which we can
apply the mixing property. Put 
$$
\epsilon_{a^-,a^+}=\epsilon_{a^-,a^+}(i,\ell)
=\sgn \left(\mathbb{E}(I_i\ind_{Z_i=\ell}\ind_{Y_i^-=a^-}\ind_{Y_i^+=a^+})
-\mathbb{E}(I_i\ind_{Z_i=\ell})\mathbb{E}(\ind_{Y_i^-=a^-}\ind_{Y_i^+=a^+})\right),
$$
and then $|B_{i,\ell}|=B_{i,\ell}^-+B_{i,\ell}^+$, where 
\[B_{i,\ell}^+=\sum_a |f(a+\ell)|\sum_{\substack{a^-+a^+=a\\ \epsilon_{a^-,a^+}=+1} }
\left(\mathbb{E}(I_i\ind_{Z_i=\ell}\ind_{Y_i^-=a^-}\ind_{Y_i^+=a^+})
-\mathbb{E}(I_i\ind_{Z_i=\ell})\mathbb{E}(\ind_{Y_i^-=a^-}\ind_{Y_i^+=a^+})\right)\]
and
$$
B_{i,\ell}^-=-\sum_a|f(a+\ell)|\sum_{\substack{a^-+a^+=a\\ \epsilon_{a^-,a^+}=-1} }
\left(\mathbb{E}(I_i\ind_{Z_i=\ell}\ind_{Y_i^-=a^-}\ind_{Y_i^+=a^+})
-\mathbb{E}(I_i\ind_{Z_i=\ell})\mathbb{E}(\ind_{Y_i^-=a^-}\ind_{Y_i^+=a^+})\right).
$$

Let us begin estimating $B_{i,\ell}^+$, the case of $B_{i,\ell}^-$ will be done below. We partition the sum over $a$ into segments of exponential progression.
For that purpose let us put, using $N=1/\mu(U)$
$$
g_m(\ell)=\max_{N2^{-m}\le a<N2^{-m+1}}|f(a+\ell)|
$$
 which by Proposition~1 of~\cite{barbour1992compound} satisfies $g_m\le c_42^m/N$ for some constant $c_4$.
 We now get
\begin{eqnarray*}
B_{i,\ell}^+&\le&\sum_{m=0}^{\lg N}g_m\sum_{a=2^{-m}N}^{2^{-m+1}N-1}
\sum_{\substack{a^-=0\\ \epsilon_{a^-,a^+}=+1}}^a
\left(\mathbb{E}(I_i\ind_{Z_i=\ell}\ind_{Y_i^-=a^-}\ind_{Y_i^+={a-a^-}})\right.\\
&&\hspace{7cm}\left.-\mathbb{E}(I_i\ind_{Z_i=\ell})\mathbb{E}(\ind_{Y_i^-=a^-}\ind_{Y_i^+= {a-a^-}})\right)\\
&\le&\sum_{m=0}^{\lg N}c_4\frac{2^m}N\sum_{a^-=0}^{2^{-m+1}N-1}
\left(\mathbb{E}(I_i\ind_{Z_i=\ell}\ind_{Y_i^-=a^-}\ind_{\mathcal{Y}^+_{i,m}(a^-)})
-\mathbb{E}(I_i\ind_{Z_i=\ell})\mathbb{E}(\ind_{Y_i^-=a^-}\ind_{\mathcal{Y}^+_{i,m}(a^-)})\right)
\end{eqnarray*}
where the set  
\begin{eqnarray*}
\mathcal{Y}^+_{i,m}(a^-)
&=&\bigcup_{\substack{2^{-m}N-a^-\le a^+<2^{-m+1}N-a^-\\ \epsilon_{a^-,a^+=+1}}}\{Y_i^+=a^+\}\\
&=&\{2^{-m}N-a^-\le Y_i^+<2^{-m+1}N-a^-: \epsilon_{a^-,a^+}=+1\}
\end{eqnarray*}
cuts out slices of exponential progression.
By the mixing property
\begin{eqnarray*}
\mathbb{E}(I_i\ind_{Z_i=\ell}\ind_{Y_i^-=a^-}\ind_{\mathcal{Y}^+_{i,m}(a^-)})
&=&\mathbb{E}(I_i\ind_{Z_i=\ell}\ind_{Y_i^-=a^-})
\mathbb{P}(\mathcal{Y}^+_{i,m}(a^-))\\
&&\hspace{2cm}+\mathcal{O}^*(\mathbb{E}(I_i\ind_{Z_i=\ell}\ind_{Y_i^-=a^-})\phi(\Delta-n))\\
&=&\mathbb{P}(Y_i^-=a^-)\mathbb{E}(I_i\ind_{Z_i=\ell})\mathbb{P}(\mathcal{Y}^+_{i,m}(a^-))\\
&&\hspace{2cm}+\mathcal{O}^*(\mathbb{P}(Y_i^-=a^-)\mathbb{P}(\mathcal{Y}^+_{i,m}(a^-))\phi(\Delta-n))\\
&&\hspace{2cm}+\mathcal{O}^*(\mathbb{E}(I_i\ind_{Z_i=\ell}\ind_{Y_i^-=a^-})\phi(\Delta-n))
\end{eqnarray*}
where the symbol $\mathcal{O}^*$ indicates an error term where the implied constant
is $1$ (i.e.\ if $G=\mathcal{O}^*(\varepsilon)$ then $|G|\le \varepsilon$).
Also
$$
\mathbb{P}(Y_i^-=a^-)\mathbb{P}(\mathcal{Y}^+_{i,m}(a^-))
=\mathbb{E}(\ind_{Y_i^-=a^-}\ind_{\mathcal{Y}^+_{i,m}(a^-)})
+\mathcal{O}^*(\mathbb{P}(Y_i^-=a^-)\phi(2\Delta-n)).
$$
Thus, for $m=0, 1, 2,\dots, \lg N$,
\begin{eqnarray*}
\sum_{a^-=0}^{2^{-m+1}N-1}
\left|\mathbb{E}(I_i\ind_{Z_i=\ell}\ind_{Y_i^-=a^-}\ind_{\mathcal{Y}^+_{i,m}(a^-)})
-\mathbb{E}(I_i\ind_{Z_i=\ell})\mathbb{E}(\ind_{Y_i^-=a^-}\ind_{\mathcal{Y}^+_{i,m}(a^-)})\right|\hspace{-12cm}&&\\
&\le&\phi(\Delta-n)\sum_{a^-=0}^{2^{-m+1}N-1}
\left(\mathbb{P}(Y_i^-=a^-)\mathbb{P}(\mathcal{Y}^+_{i,m}(a^-)
+\mathbb{E}(I_i\ind_{Z_i=\ell}\ind_{Y_i^-=a^-})
+\mathbb{E}(I_i\ind_{Z_i=\ell})\mathbb{P}(Y_i^-=a^-)\right)\\
&\le&3\phi(\Delta-n)\sum_{a^-=0}^{2^{-m+1}N-1}
\mathbb{E}(I_i\ind_{Z_i=\ell}\ind_{\mathcal{Y}^+_{i,m}(a^-)})\\
&\le&3\phi(\Delta-n)\mathbb{P}(Z_i=\ell, I_i=1)
\end{eqnarray*}
and 
$$
B_{i,\ell}^+\le 2c_5 \phi(\Delta-n)\sum_{m=0}^{\lg N}\frac{2^m}N
\le 2c_5\phi(\Delta-n).
$$
Similarly one estimates the negative term $B_{i,\ell}^-$
which yields the estimate $B_{i,\ell}^-\le c_5\phi(\Delta-n)$.
Along the way one uses the set
$$
\mathcal{Y}^-_{i,m}(a^-)
=\bigcup_{\substack{2^{-m}N-a^-\le a^+<2^{-m+1}N-a^-\\ \epsilon_{a^-,a^+=-1}}}\{Y_i^+=a^+\},
$$
where we note that
$$
\mathcal{Y}^-_{i,m}(a^-)\cup\mathcal{Y}^+_{i,m}(a^-)=\{2^{-m}N-a^-\le Y_i^+<2^{-m+1}N-a^-\}.
$$
Consequently
\begin{equation}\label{eq:B}
|B|
\le\sum_{i=0}^{\frac{t}{\mu(U)}}\sum_\ell(B_{i,\ell}^-+B_{i,\ell}^+)
\le\frac{4tc_5}{\mu(U)} \sum_\ell{\phi(\Delta-n)}\mathbb{E}(I_i\ind_{Z_i=\ell})
\le \frac{c_6 t\phi(\Delta-n)}{\mu(U)}.
\end{equation}

\item  In order to estimate $C$ we proceed in a similar way as in part~(ii). Indeed one has 
$$
|\mathbb{E}(f(Y_i+\ell)-f(W+\ell))|
\le\|f'\|\mathbb{E}(I_i+Z_i+V_i)
\le c_7 \Delta\mu(U).
$$
Thus
\begin{equation}\label{eq:C}
|C|\le\sum_i\sum_\ell c_7\Delta\mu(U)\mathbb{E}(I_i\ind_{Z_i=\ell})
\le tc_7 \Delta\mathbb{E}(I_i\ind_{Z_i\ge1})
\le c_8 t\Delta \mu(U).
\end{equation}
\end{enumerate}

Combining~\eqref{eq:A}, \eqref{eq:B} and \eqref{eq:C} we finally achieve that, for any  $K<\Delta<t/\mu(U)$
$$
|\mu(W\in E)-\tilde\nu(E) |
\le C_1 t\!\left(K\frac{\phi(\Delta-n)}{\mu(U)}+\Delta\mu(U)+\phi^1(K/2)+\sum_{j=K/2}^n\mu(U^j)\right),
$$
for some $C_1\ge c_2c_3+c_6+c_8$ thus proving~\eqref{eq:th1} {in the right $\phi$-mixing case.}

{To get the corresponding conclusion in the left $\phi$-mixing case notice that in the estimates of $A_{i,\ell}$ and 
$A$ we have to replace $U^j$ by $\tilde{U}^j$ and obtain that the estimate~\eqref{eq:A} is modified to 
$$
|A|\le c_2c_3 t\left[\Delta\mu(U)+\phi^1(K/2)+\sum_{j=K/2}^n\mu(\tilde{U}^j)\right].
$$
The estimates that lead to the bound of the term $B$ will be the same although the order of 
splitting and combining terms is the reverse.}
\end{proof}

\begin{remark}\label{rem:bothsides}
{The only part of the proof of Theorem~\ref{theorem.compound.poisson} where right $\phi$-mixing property is required is for the estimate of $A_{i,\ell}$ (specifically display \eqref{eq:critical_part_mixing}).} Suppose we are in the invertible case and $\Omega$ is a shift space with 
map $T=\sigma$ the left shift map. If we were to use the left $\phi$-mixing property then the sets 
$U^i$ would have to be replaced by the set $\tilde{U}^i=\sigma^{-(n-i)}A_i(\sigma^{n-i}U)$.
Now we can take the sets $U$ to be $n$-approximation of an unstable leaf $\Gamma$ through a point $x\in\Omega$
e.g.\ $\Gamma=\{y\in\Omega: y_i=x_i\;\forall\;i\le0\}$. Obviously $\Gamma$ is a nullset
but in this case we get that $\tilde{U}^i=\Omega$ the entire space whenever $i<n/2$.
The right $\phi$-mixing property avoids this problem. {This also explains how the proof has to be changed if we assume left $\phi$-mixing instead of right $\phi$-mixing: the only difference is in display \eqref{eq:critical_part_mixing} where we have to use $\tilde{U}^{j/2}$ instead of $U^{j/2}$.}
\end{remark}

\subsection{Proof of Theorem~\ref{corollary.compound.poisson_psi}}

This proof is similar to the one of Theorem~\ref{theorem.compound.poisson}
but allows for some simplifications which we outline below.

\begin{proof}[Proof of Theorem~\ref{corollary.compound.poisson_psi}]
As above let $\tilde\nu$ be the compound Poisson distribution for $t{\tilde\lambda_\ell}$, $\ell\in\mathbb{N}$ as defined in the statement of the corollary and the preceeding theorem. Also we let $E\subset{N}_0$ and $f$ the solution of the Stein equation $\mathscr{S}f=\ind_E-\tilde\nu(E)$.

For $K<\!\!<\Delta<\!\!<t/\mu(U)$ we denote as above by $Z_i$ the close range interactions,
by $V_i^\pm$ the gap terms and by $Y_i^\pm$ the two halves of the principal terms.

As in the proof of the theorem we split the error into three parts:
$$
\mu(W\in E)-\tilde\nu(E)=A+B+C,
$$
where $A$ and $C$ cover short term gap interactions in the dependent and independent case and 
$B$ is the error that comes from the principal term with long range interactions.
We now proceed to show that each of the three terms can be upper bounded in order to give the  stated bounded. 
\begin{enumerate}[(i)]
\item For the first term we write $A=\sum_{i,\ell}A_{i,\ell}$ where
$$
|A_{i,\ell}|=|\mathbb{E}[I_i\ind_{Z_i=\ell}(f(Y_i+V_i+\ell))-f(Y_i+\ell)]|
\le\|f'\|\mathbb{E}[I_i\ind_{Z_i=\ell}V_i].
$$
The $\psi$-mixing property then yields ($U^\ell=A_\ell(U)$)
\begin{eqnarray*}
\sum_\ell\mathbb{E}(I_i\ind_{Z_i=\ell}V_i^+)
&\le&\sum_{j=K+1}^{2n}\mu(U^{j/2}\cap T^{-j}U)
+\sum_{j=2n+1}^{K+\Delta}\mu(U\cap T^{-j}U)\\
&\le&\mu(U)\left(\sum_{j=K+1}^{2n}\mu(U^{j/2})(1+\psi(j/2))
+\sum_{j=2n+1}^{K+\Delta}\mu(U)(1+\psi(j-n))\right)\\
&\le&\mu(U)\left(2\sum_{j=K/2}^{n}\mu(U^{j})+2\Delta\mu(U)\right)
\end{eqnarray*}
for $K$ large enough,
and similarly for the left part of the gap $V_i^-$.
Thus, since $\|f'\|=\mathcal{O}(1)$ by Theorem~4 in~\cite{barbour1992compound},
$$
|A|\le c_1  t\left(\sum_{j=K/2}^{n}\mu(U^{j})+2\Delta\mu(U)\right).
$$

\item We now estimate  $B$ which we split as before 
$$
B=\sum_{i}\sum_\ell B_{i,\ell},
$$ 
$|B_{i,\ell}|=B_{i,\ell}^-+B_{i,\ell}^+$, where for $\epsilon=+1,-1$:
$$
B_{i,\ell}^\epsilon=\epsilon\sum_a |f(a+\ell)|\sum_{\substack{a^-+a^+=a\\  \epsilon_{a^-,a^+}=\epsilon} }
\left(\mathbb{E}(I_i\ind_{Z_i=\ell}\ind_{Y_i^-=a^-}\ind_{Y_i^+=a^+})
-\mathbb{E}(I_i\ind_{Z_i=\ell})\mathbb{E}(\ind_{Y_i^-=a^-}\ind_{Y_i^+=a^+})\right)
$$
with
$$
\epsilon_{a^-,a^+}=\epsilon_{a^-,a^+}(i,\ell)
=\sgn \left(\mathbb{E}(I_i\ind_{Z_i=\ell}\ind_{Y_i^-=a^-}\ind_{Y_i^+=a^+})
-\mathbb{E}(I_i\ind_{Z_i=\ell})\mathbb{E}(\ind_{Y_i^-=a^-}\ind_{Y_i^+=a^+})\right).
$$

Let us begin estimating $B_{i,\ell}^+$, the case of $B_{i,\ell}^-$ will be done below. We partition the sum over $a$ into segments of exponential progression.
For that purpose let us put, using $N=1/\mu(U)$
$$
g_m(\ell)=\max_{N2^{-m}\le a<N2^{-m+1}}|f(a+\ell)|
$$
 which by Proposition~1 of~\cite{barbour1992compound} satisfies $g_m\le c_12^m/N$ for some constant $c_1$. We now get
$$
B_{i,\ell}^\epsilon
\le\epsilon\sum_{m=0}^{\lg N}c_1\frac{2^m}N\sum_{a^-=0}^{2^{-m+1}N-1}
\left(\mathbb{E}(I_i\ind_{Z_i=\ell}\ind_{Y_i^-=a^-}\ind_{\mathcal{Y}^+_{i,m}(a^-)})
-\mathbb{E}(I_i\ind_{Z_i=\ell})\mathbb{E}(\ind_{Y_i^-=a^-}\ind_{\mathcal{Y}^+_{i,m}(a^-)})\right)
$$
where as above
\begin{eqnarray*}
\mathcal{Y}^+_{i,m}(a^-)
&=&\bigcup_{\substack{2^{-m}N-a^-\le a^+<2^{-m+1}N-a^-\\ \epsilon_{a^-,a^+=+1}}}\{Y_i^+=a^+\}\\
&=&\{2^{-m}N-a^-\le Y_i^+<2^{-m+1}N-a^-: \epsilon_{a^-,a^+}=+1\}
\end{eqnarray*}
cuts out slices of exponential progression.
By the $\psi$-mixing property
\begin{eqnarray*}
\mathbb{E}(I_i\ind_{Z_i=\ell}\ind_{Y_i^-=a^-}\ind_{\mathcal{Y}^+_{i,m}(a^-)})
&=&\mathbb{E}(I_i\ind_{Z_i=\ell}\ind_{Y_i^-=a^-}))
\mathbb{P}(\mathcal{Y}^+_{i,m}(a^-))(1+\mathcal{O}^*(\psi(\Delta-n)))\\
&=&\mathbb{P}(Y_i^-=a^-)\mathbb{E}(I_i\ind_{Z_i=\ell})\mathbb{P}(\mathcal{Y}^+_{i,m}(a^-))(1+\mathcal{O}^*(\psi(\Delta-n)))\\
&=&\mathbb{E}(I_i\ind_{Z_i=\ell})\mathbb{E}(\ind_{Y_i^-=a^-}\ind_{\mathcal{Y}^+_{i,m}(a^-)})(1+\mathcal{O}^*(\psi(\Delta-n)))
\end{eqnarray*}
where the symbol $\mathcal{O}^*$ indicates an error term where the implied constant
is $1$ (i.e.\ if $G=\mathcal{O}^*(\varepsilon)$ then $|G|\le \varepsilon$).
Thus
\begin{eqnarray*}
B_{i,\ell}^\epsilon
&\le&c_1\sum_{m=0}^{\lg N}\frac{2^m}N\sum_{a^-=0}^{2^{-m+1}N-1}
\mathbb{E}(\ind_{Y_i^-=a^-}\ind_{\mathcal{Y}^+_{i,m}(a^-)})\mathbb{E}(I_i\ind_{Z_i=\ell})\psi(\Delta-n)\\
&\le& c_1\psi(\Delta-n)\sum_{m=0}^{\lg N}\frac{2^m}N \mathbb{E}(I_i\ind_{Z_i=\ell})\\
&\le&  2c_1\psi(\Delta-n)\mathbb{E}(I_i\ind_{Z_i=\ell}) 
\end{eqnarray*}
and consequently 
$$
|B|
\le 2 c_1\psi(\Delta-n)\sum_{i=1}^N\sum_{\ell=1}^{2K+1}\mathbb{E}(I_i\ind_{Z_i=\ell}) 
\le 2c_1\psi(\Delta-n)N\mathbb{E}(I_i\ind_{Z_i\ge1})
\le 2c_3t \psi(\Delta-n).
$$

\item  The estimate of the term $C$ is exactly the one from Theorem~\ref{theorem.compound.poisson},
\[
|C|\le c_2 t\Delta \mu(U).
\]
\end{enumerate}

Combining the estimates we end up with 
$$
|\mu(W\in E)-\tilde\nu(E) |
\le C_1't \inf_{K<\Delta<t/\mu(U)} \!\left(\psi(\Delta-n)+\Delta\mu(U)+\sum_{j=K/2}^n\mu(U^j)\right)
$$
for some constant $C_1'$.
\end{proof}

\subsection{Proof of Theorem~\ref{main.theorem}}
Recall that we start with a nested sequence of sets $U_n,n\ge1$. For $K<t/\mu(U_n)$ we define
 $Z_i^+=\sum_{j=0}^KI_{i+j}$ and $Z_i^-=\sum_{j=1}^KI_{i-j}$, where we assume that $i\ge K$. Let us also define $W^L=\sum_{\ell=0}^L I_\ell$. In order to prove this result we first state the following lemma taken from~\cite{haydn2020limiting}.  

\begin{lemma} \label{tail.lemma}
Assume that the limits $\alpha_k,k\ge1$ (see \eqref{eq:alpha}) exist and furthermore  $\sum_{k=1}^\infty k^2\alpha_k<\infty$. Then for every $\eta>0$ there exists an $L_0$ so that for all $L\ge L_0$:
 $$
\left|\mathbb{E}(\ind_{Z_i^+=k}\ind_{Z_i^-=\ell-k}I_i)-\mathbb{E}(\ind_{Z_i^+=k'}\ind_{Z_i^-=\ell-k'}I_i)\right|
\le\eta\mu(U)
$$
for all $n$ large enough (depending on $L,\ell$).
\end{lemma}
We are now ready to prove the theorem.

\begin{proof}[Proof of Theorem~\ref{main.theorem}]
 For $E\subset \mathbb{N}_0$ and $K<t/\mu(U_n)$
\begin{equation}\label{eq:triangular2}
|\mu(W_n\in E)-\tilde\nu(E)|\le |\mu(W_n\in E)-\tilde\nu_{K,U_n}(E)|+|\tilde\nu_{K,U_n}(E)-\tilde\nu(E)|
\end{equation}
where $\tilde\nu_{K,U_n}$ is as in the statement of Theorem \ref{theorem.compound.poisson}. In order to prove Theorem~\ref{main.theorem} it is therefore enough to prove that both terms on the RHS converge to $0$ as $n\rightarrow0$ and $K\rightarrow\infty$. We proceed in two steps. 
 \begin{enumerate}
\item We start proving that the second term on the RHS of \eqref{eq:triangular2} converges to 0.  First recall the definitions of $\alpha_k(L,U_n),\alpha_k(L)$ and $\alpha_k$ given in \eqref{eq:alpha} and that of $\lambda_\ell(K,U_n)$ given in $\eqref{eq:tildelambda}$.  We have that $\tilde\nu$ and $\tilde\nu_{K,U_n}$ are Poisson compounds with parameters $\tilde\lambda_\ell:=\alpha_k-\alpha_{k+1}$ and $\lambda_\ell(K,U_n)$ respectively. 
So what has to be proved is that, provided $\alpha_k$ exists, we have 
$\alpha_k-\alpha_{k+1}=\lim_{K\to\infty}\lim_{n\to\infty}\lambda_\ell(K,U_n)$, that is, the convergence of the parameters of the involved Poisson compounds distributions.

Observe that $\alpha_k(K,U)$ can be written as $\mathbb{E}(\ind_{Z_0=k}|I_0)$. On the other hand,  by translation invariance, $\mathbb{E}(\ind_{Z_0=k}I_0)=\mathbb{E}(\ind_{Z_i^+=k}I_i)$ for any $i\ge1$. We therefore work on the later quantity. 
Consider the disjoint union
$$
\{Z_i^+=k\}\cap \{I_i=1\}=\bigcup_{\ell=k}^\infty\{Z_i^+=k\}\cap\{Z_i^-=\ell-k\}\cap\{I_i=1\}.
$$
By invariance, the expectations
$\mathbb{E}(\ind_{Z_i^+=k}\ind_{Z_i^-=\ell-k}|I_i=1)$
are equal for all $i$. Let us note that in the conditions of Theorem \ref{main.theorem}, we have $\sum_{k\ge1}k^2\alpha_k<\infty$. Thus we can use  Lemma~\ref{tail.lemma} which states that if $\eta>0$, then for all $K$ large enough
$$
\left|\mathbb{E}(\ind_{Z_i^+=k}\ind_{Z_i^-=\ell-k}I_i)-\mathbb{E}(\ind_{Z_i^+=k'}\ind_{Z_i^-=\ell-k'}I_i)\right|
\le\eta\mu(U)
$$
for $k,k'=1,2\dots,\ell$. Hence, since $Z_i=Z_i^-+Z_i^+$
$$
\mathbb{E}(\ind_{Z_i^+=k}\ind_{Z_i^-=\ell-k}I_i)
=\frac1\ell\mathbb{E}(\ind_{Z_i=\ell}I_i)(1+\mathcal{O}(\eta))
=\tilde\lambda_\ell(K,U_n)\mu(U)(1+\mathcal{O}(\eta))
$$
and therefore
$$
\mathbb{E}(\ind_{Z_i^+=k}I_i)
=\sum_{\ell=k}^\infty\mathbb{E}(\ind_{Z_i^+=k}\ind_{Z_i^-=\ell-k}I_i)
=(1+\mathcal{O}(\eta))\mu(U)\sum_{\ell=k}^\infty\tilde\lambda_\ell(K,U_n).
$$
According to what we said above, we therefore have
\begin{eqnarray*}
\alpha_k(K,U)&=&\mathbb{E}(\ind_{Z_0^+=k}|I_0)\\
&=&\frac{\mathbb{E}(\ind_{Z_0^+=k}I_0)}{\mu(U_n)}=\frac{\mathbb{E}(\ind_{Z_i^+=k}I_i)}{\mu(U_n)}\\
&=&(1+\mathcal{O}(\eta))\sum_{\ell=k}^\infty\tilde\lambda_\ell(K,U_n).
\end{eqnarray*}
So in particular $\alpha_k(K,U_n)-\alpha_{k+1}(K,U_n)=(1+\mathcal{O}(\eta))\tilde\lambda_k(K,U_n)$, valid for any positive $\eta\rightarrow0$, thus provided the limit $\alpha_k$ exists, we have $\lim_K\lim_n\tilde\lambda_k(K,U_n)= \alpha_k-\alpha_{k+1}$.

\item In order to prove that the first term of the RHS of \eqref{eq:triangular2} converges to 0 we naturally use Theorem \ref{theorem.compound.poisson}.
Let $\beta\in(0,1)$ and choose $\Delta=\mu(U_n)^{-\beta}$, we get by \eqref{eq:th1} that $|\mu(W_n\in E)-\tilde\nu_{K,U_n}(E) |$ is bounded above by
$$
 C_1t \!\left(K\frac{\phi(\mu(U_n)^{-\beta})}{\mu(U)}
 +\mu(U_n)^{1-\beta}
 +\phi^1(K/2)
 +\sum_{j=K/2}^n\mu(U_n^j)\right)
$$
where we recall that, by assumption, for any sufficiently large $n$'s, the fourth term is bounded above by   $a_{K/2}$. The two first terms  go to zero as $n$ diverges with a suitable choice of $\beta<1$ so that $\beta>\frac1\gamma$. Then, taking $K\to\infty$ we get by assumption that $a_{K/2}$
and $\phi^1(K/2)$ vanish as well ({by summability of $\phi$)}. This concludes the proof of the theorem. 
\end{enumerate}
\end{proof}


\section{Proofs of the results of Sections \ref{sec:discussion} and \ref{sec:examples2}}\label{sec:remaining_results}

\begin{proof}[Proof of Proposition \ref{prop:rege_phi}]
 For any two measurable sets $A\in\sigma(X_0,\ldots,X_{n-1})$ and $B\in\sigma(X_{n+k-1}^\infty)$ with positive probability, let 
 \[
\phi_{A,B}(k):= |\mathbb P(X\in B|X\in A)-\mathbb P(X\in B)|.
 \]
Defining $l(n):=\inf\{j\ge1:T_j\ge n\}$ and using the regenerative property, we get
\begin{align*}
\phi_{A,B}(k)&=\left|\sum_{i\ge1}\mathbb P(T_{l(n)}-n=i,X\in B|X\in A)- \mathbb P(X\in B)\right|\\
&=   \left|\sum_{i> k}\mathbb P(T_{l(n)}-n=i,X\in B|X\in A)+\sum_{i=1}^{k}\mathbb P(T_{l(n)}-n=i,X\in B)\right.\\&\,\,\,\,\,\,\,-\sum_{i\ge1}\mathbb P(T_{l(n)}-n=i,X\in B)\Big|\\
&=   \left|\sum_{i> k}\mathbb P(T_{l(n)}-n=i,X\in B|X\in A)-\sum_{i>k}\mathbb P(T_{l(n)}-n=i,X\in B)\right|\\
&\le  \sum_{i> k}\left|\mathbb P(T_{l(n)}-n=i,X\in B|X\in A)-\mathbb P(T_{l(n)}-n=i,X\in B)\right|\\
&\le  \sum_{i> k}\left(\mathbb P(T_{l(n)}-n=i|X\in A)+\mathbb P(T_{l(n)}-n=i)\right).
\end{align*}
We now consider the particular case when $\{X\in A\}=\bigcap_{i=0}^{n-1}\{X_i=b_i\}$ for some string $b_0,\ldots,b_{n-1}$ of symbols of $\mathcal A$. We get
\begin{align*}
\phi_{A,B}(k)&\leq \sum_{i> k}\left(\mathbb P(T_{1}=i|X_0=b_{n-1})+\mathbb P(T_{1}=i)\right)\\
&=  \sum_{i> k}\left(\bar q_{b_{n-1}}(i)+\sum_a\bar p(a)\bar q_a(i)\right)\\
&\le 2 \sup_{a\in\mathcal A}\sum_{i> k}\bar q_{a}(i).
\end{align*}
\end{proof}

\begin{proof}[Proof of Proposition \ref{prop:polya_gmeasure}]
In the symbolic setting, if $x$ has prime period $m$, we can write $x=s^{\infty}$, the concatenation of infinitely many times a fixed string $s:=s_1\ldots s_m$ in which $s_i\in \mathcal A$ for $i=1,\ldots m$. In this case, let $k_n:=\lfloor n/m\rfloor$ and $r_n:=n-mk_n$, then
\begin{align*}
A_{n}(x)&=s^{k_n}\,s_{1}^{r_n}\\
A_{n+m}(x)&=s^{k_n}\,s_{1}^{r_n}\,s_{r_n+1}^{m}s_{1}^{r_n}
\end{align*}
where we used the shorthand notation $s_i^j:=s_i\ldots s_j$ for $i\le j$ and where we also write
$s^k$ for the word $s$ repeated $k$ times.
So 
\begin{align*}
\frac{\mu(A_{n+m}(x))}{\mu(A_n(x))}&=\frac{\mu([s^{k_n}\,s_{1}^{r_n}\,s_{r_n+1}^{m}s_{1}^{r_n}])}{\mu([s^{k_n}\,s_{1}^{r_n}])}=\frac{\mu([ss^{k_n}s_{1}^{r_n}])}{\mu([s^{k_n}\,s_{1}^{r_n}])}
\end{align*}
with the convention that $a_1^0=\emptyset$ (to include the case $r_n=0$).
The later ratio equals 
\begin{equation*}
\frac{\mu([s_1s_2^{m-1}s^{k_n-1}s_{1}^{r_n}])}{\mu([s_2^{m-1}s^{k_n-1}s_{1}^{r_n}])}\frac{\mu([s_2^{m-1}s^{k_n-1}s_{1}^{r_n}])}{\mu([s_3^{m-1}s^{k_n-1}s_{1}^{r_n}])}\ldots\frac{\mu([s_ms^{k_n}\,s_{1}^{r_n}])}{\mu([s^{k_n}\,s_{1}^{r_n}])},
\end{equation*}
which can be re-written as 
\begin{equation}\label{eq:sequences_limit}
\frac{\mu([x_1^{n+m}])}{\mu([x_2^{n+m}])}\frac{\mu([x_2^{n+m}])}{\mu([x_3^{n+m}])}\ldots\frac{\mu([x_m^{n+m}])}{\mu([x_{m+1}^{n+m}])}.
\end{equation}
So, for  $\lim_{n\to\infty}\frac{\mu(A_{n+m}(x))}{\mu(A_n(x))}$ to exist, it is enough that the limits of the  $m$ terms which are multiplied above, exist. 

But observe that, if $x$ is a continuity point for $g$, then 
\[
\text{var}_kg:=\sup\{|g(y)-g(z)|:y,z\in[x_1^k]\}\rightarrow0.
\]
Thus we can write
\begin{equation}\label{eq:variation_cylinder}
\mu([x_1^{n}])=\int_{[x_2^n]} g(x_1y)d\mu(y)=\mu([x_2^n])[g(x)+\mathcal O^*(\text{var}_{n-1}\,g
)]
\end{equation}
meaning that, for continuity points $x$, we have that $\mu(x_1^{n})/\mu(x_2^{n})$ converges to $g(x)$ as $n$ diverges

Coming back to \eqref{eq:sequences_limit}, due to the fact that $g$ is continuous at $x,\sigma(x),\ldots,\sigma^{m-1}(x)$, we conclude that the product converges to $\prod_{i=0}^{m-1}g(\sigma^{i}x)$, concluding the proof of existence and computation of $p$.


The existence of the limiting parameters is now proved, and according to the discussion of Subsection \ref{sec:examples}, the proof of the second statement follows automatically using our Theorem \ref{main.theorem}. Indeed, the  assumptions of this theorem are granted since as we already said, under summable variation, the measure is $\psi$-mixing, and moreover, the assumption that $g>0$ implies that $U_n^j=A_j(x)$ has exponentially decaying measure in $j$ since a simple argument shows that $\mu(A_j(x))\le (\sup g)^j$. 
\end{proof}

\begin{proof}[Proof of Proposition \ref{prop:polya_renewal}]
All the properties we use here, concerning the renewal measure, are proved in \cite{abadi/cardeno/gallo/2015}. First, the existence of the parameters follows from Theorem 3.2 therein. Under our conditions, the measure under study is left $\phi$-mixing with exponentially decaying rate $\phi$. Since the map is not invertible, by Remark \ref{rem:bothsides}, we only need left $\phi$-mixing, but in any case, the renewal measure is reversible, and therefore it enjoys both, left and right $\phi$-mixing with the same rate. Finally, since $p_i\in[\epsilon,1-\epsilon]$, the same holds for $g$, which automatically implies that $U_n^j=A_j(x)$ has exponentially decaying measure in $j$ as in the preceding proof. 
\end{proof} 

\begin{proof}[Proof of Proposition \ref{prop:furst}]
First of all, let us observe that the example is  $\psi$-mixing as it is a two-coordinates factor map of a product measure.  So it satisfies  condition (1) of Theorem \ref{main.theorem} concerning the mixing properties. The measure of cylinders of size $n$ decays exponentially fast, so the second condition of Theorem \ref{main.theorem} is granted as well. It only remains to check the third condition, but it holds if we are able to prove that the limit  \eqref{eq:limit} holds. This is what we prove below. 

Consider a point $y\in\{-1,+1\}^{\mathbb N}$  of prime period $m\ge1$. By \eqref{eq:sequences_limit} we have to compute the limit of
\[
\frac{\nu([y_1^{n+m}])}{\nu([y_2^{n+m}])}\frac{\nu([y_2^{n+m}])}{\nu([y_3^{n+m}])}\ldots\frac{\nu([y_m^{n+m}])}{\nu([y_{m+1}^{n+m}])}.
\]
Let us start computing (and prove it exists) the limit of
\[
\frac{\nu([1y_2^{n}])}{\nu([y_{2}^{n}])},n\ge1.
\]

Observe   that $x^+(y_{2}^\infty)$ and $x^-(y_{2}^\infty)$ are also periodic points. Let us denote by $m'$ their common prime period. Denote by $S_n$ the number of ones in 
\[
x^+(y_2^\infty)_2,x^+(y_2^\infty)_3,\ldots,x^+(y_2^\infty)_{n+2}
\]
(that is, of the $n+1$ first coordinates of $x^+(y_2^\infty)$). For technical matters, we will write $n$ as $k_n(2m'-1)+r_n$ where $k_n:=\lfloor n/(2m'-1)\rfloor$ and $r_n$ is the remaining part, strictly smaller than $2m'-1$. Observe that, since $2m'$ is a period of $x^+(y_2^\infty)$, we have $S_n=k_nS_{2m'-1}+R_n$ where $R_n:=S_n-k_nS_{2m'-1}<S_{2m'-1}$. Therefore
\[
\frac{S_n}{n+1}=\frac{k_nS_{2m'-1}+R_n}{k_n(2m'-1)+r_n+1}\rightarrow \frac{S_{2m'-1}}{2m'-1}.
\]

On the other hand, a simple calculation (see \cite{ferreira2020non}) gives that 
\[
\frac{\nu([1y_2^{n+1}])}{\nu([y_2^{n+1}])}=(1-\epsilon)\frac{\left(\frac{\epsilon}{1-\epsilon}\right)+\left(\frac{\epsilon}{1-\epsilon}\right)^{2(n+1)\left(\frac12-\frac{S_n}{n+1}\right)}}{1+\left(\frac{\epsilon}{1-\epsilon}\right)^{2(n+1)\left(\frac12-\frac{S_n}{n+1}\right)}}.
\]
We therefore have the following limits according to the values of $\epsilon$ and $\frac{S_{2m'-1}}{2m'-1}$: 
\begin{equation}\label{eq:conv_furst2}
\frac{\nu([1y_2^n])}{\nu([y_2^n])}\stackrel{n\rightarrow\infty}{\longrightarrow} \left\{\begin{array}{ccc}
		\epsilon&\text{ if }0<(\frac12-\epsilon)(\frac12-\frac{S_{2m'-1}}{2m'-1})\\
		1-\epsilon&\text{ if }0>(\frac12-\epsilon)(\frac12-\frac{S_{2m'-1}}{2m'-1})\\
	\end{array}\right.
\end{equation}
(Obviously, $\lim\frac{\nu([(-1)y_2^n])}{\nu([y_2^n])}=1-\lim\frac{\nu([1y_2^n])}{\nu([y_2^n])}$.) 
The same limiting value \eqref{eq:conv_furst2} holds for 
\[
\frac{\nu([1y_i^{n+m}])}{\nu([y_{i+1}^{n+m}])},i=3,\ldots,m.
\]
 So let $k$ be the number of $+1$'s in the period  of $y$ (which we recall is of size $m$).  According to \eqref{eq:sequences_limit} we can thus conclude, 
\begin{equation}\label{eq:limitingFurst}
p^{(2)}_m=\left\{\begin{array}{ccc}
		\epsilon^{k}(1-\epsilon)^{m-k}&\text{ if }0<(\frac12-\epsilon)(\frac12-\frac{S_{2m'-1}}{2m'-1})\\
                 (1-\epsilon)^{k}\epsilon^{m-k}&\text{ if }0>(\frac12-\epsilon)(\frac12-\frac{S_{2m'-1}}{2m'-1}).\\
\end{array}\right.
\end{equation}
\end{proof}

\begin{proof}[Proof of the first statement of Theorem \ref{theo:tempsync}] 
For simplicity, we do the proof with $m=2$, but the general case follows identically. By assumption $\hat\mu$ is $\psi$-mixing and consequently the conditions (1) and (2) of  Theorem \ref{main.theorem} are satisfied. So if we prove that $\hat\alpha_{k+1}$ exists for any $k$ and satisfies $\hat\alpha_{k+1}=p^k$ for some $p\in(0,1)$,
then we prove at once that Theorem \ref{main.theorem} holds and that the asymptotic distribution is Polya-Aeppli as stated.  
%

Write
\[
\hat\alpha_{k+1}:=\frac{\hat\mu(\bigcap_{i=0}^{k}\hat \sigma^{-i}S_n)}{\hat\mu(S_n)}=\frac{\hat\mu(S_n\cap \hat \sigma^{-1}S_n)}{\hat\mu(S_n)}\frac{\hat\mu(S_n\cap \hat \sigma^{-1}S_n\cap \hat \sigma^{-2}S_n)}{\hat\mu(S_n\cap \hat \sigma^{-1}S_n)}\ldots\frac{\hat\mu(\bigcap_{i=0}^{k}\hat \sigma^{-i}S_n)}{\hat\mu(\bigcap_{i=0}^{k-1}\hat \sigma^{-i}S_n)}
\] 
which by translation invariance writes
\begin{align*}
\frac{\hat\mu(S_n\cap \hat \sigma^{-1}S_n)}{\hat\mu(\hat \sigma^{-1}S_n)}\frac{\hat\mu(S_n\cap \hat \sigma^{-1}S_n\cap \hat \sigma^{-2}S_n)}{\hat\mu(\hat \sigma^{-1}S_n\cap \hat \sigma^{-2}S_n)}
\ldots\frac{\hat\mu(\bigcap_{i=0}^{k}\hat \sigma^{-i}S_n)}{\hat\mu(\bigcap_{i=1}^{k}\hat \sigma^{-i}S_n)}.
\end{align*}
Now,  for $j=1,\ldots,k$ and $n$ large enough
\[
\frac{\hat\mu(\bigcap_{i=0}^{j}\hat \sigma^{-i}S_n)}{\hat\mu(\bigcap_{i=1}^{j}\hat \sigma^{-i}S_n)}
=\frac{\hat\mu((x,y):x_1^{n+j}=y_1^{n+j})}{\hat\mu((x,y):x_2^{n+j}=y_2^{n+j})}
=:u_{n+j}
\]
Let us assume for now that $u_{n+1}=\frac{\hat\mu(S_n\cap \hat \sigma^{-1}S_n)}{\hat\mu(S_n)}$ converges, and let $p$ denote the limit. Then for any $j\ge1$, $u_{n+j}\rightarrow p$ and the limit defining $\hat\alpha_{k+1}$ exists and it equals $p^k$. 
In other words, provided the limit $p$ exists we always have, in the limit, a P\'olya-Aeppli distribution with parameter $t(1-p)$, as stated by the theorem.

So it only remains to prove the existence of the limit $p$. Consider the projection operator $\Pi:\Sigma\time\Sigma\rightarrow\{0,1\}^\mathbb{N}$ defined through $\Pi(x,y)=z$ where $z_i=\ind_{x_i=y_i}$.  With this we now have to check whether 
\[
\lim_n\frac{\hat\mu\circ\Pi^{-1}([1^{n+1}])}{\hat\mu\circ\Pi^{-1}([1^{n}])}=\lim_n\mathbb{E}_{\hat\mu\circ\Pi^{-1}}(\ind_{[1]}|\mathcal{F}_2^n)(1^{(\infty)})
\] 
exists. Using \cite[Proposition 5]{PPW82} we only have to prove that the measure $\hat\mu\circ\Pi^{-1}$ has a continuous and strictly positive $g$-function. By assumption, $\hat g$ is strictly positive and with summable variation.  By Theorem 1.1 of \cite{verbitskiy2011factors}, we automatically have that $\hat\mu\circ\Pi^{-1}$ has an everywhere continuous and strictly positive $g$-function. This concludes the proof of the theorem. 
\end{proof}

\begin{proof}[Proof of the second statement of Theorem \ref{theo:tempsync}]

For simplicity, we do the proof with $m=2$, but the general case follows identically. 
For that reason let $\hat\mu$ be the $\hat g$ measure on $\Sigma_B\times\Sigma_B$. As above, the first two conditions of Theorem \ref{main.theorem} are granted under our assumptions. We will show that $\hat\alpha_{2}$ exists by computing it, this will automatically grant $\hat\alpha_{k+1}=p^k$ and the third condition of Theorem \ref{main.theorem}, and conclude our proof. 

By conformality we have then for all finite words $\alpha, \beta$ that  
$$
\hat\mu(\sigma[\alpha]\times\sigma[\beta])=\int_{[\alpha]\times[\beta]}\hat{g}(x,y)^{-1}\,d\hat\mu(x,y).
$$
In particular, if we put $\hat{g}_k(x,y)=\prod_{j=0}^{k-1}\hat{g}(\sigma^j(x),\sigma^j(y))$,
then for $k$-words $\alpha',\alpha''$ and $n$-words $\beta',\beta''$ one has
$$
\hat\mu([\alpha'\beta']\times[\alpha''\beta''])
=\hat\mu([\beta']\times[\beta''])\hat{g}_k(\alpha'\beta',\alpha''\beta'')e^{\mathcal{O}(v^1_n)},
$$
where $v^1_n=\sum_{j=n}^\infty v_j$ is the tailsum of  $v_n=\var_n\,\hat{g}$ and 
$\hat{g}(\gamma',\gamma'')=\sup_{(x,y)\in[\gamma']\times[\gamma'']}\hat{g}(x,y)$.

By assumption $\hat\mu$ is $\psi$-mixing and consequently the conditions of  Theorem \ref{main.theorem} 
are satisfied if we prove that the following limit exists
\[
\hat\alpha_{k+1}=\lim_{n\to\infty}\frac{\hat\mu(\bigcap_{i=0}^{k}\hat \sigma^{-i}S_n)}{\hat\mu(S_n)}.
\]
Indeed
\begin{eqnarray*}
\hat\mu(\bigcap_{i=0}^{k}\hat \sigma^{-i}S_n)
&=&\hat\mu(S_{n+k})\\
&=&\sum_{\alpha\in\mathcal A^k}\sum_{\beta\in \mathcal A^n}\hat\mu([\alpha\beta]\times[\alpha\beta])\\
&=& \sum_{\beta\in \mathcal A^n}\hat\mu([\beta]\times[\beta])
\sum_{\alpha\in\mathcal A^k}\hat{g}_k(\alpha\beta,\alpha\beta)e^{\mathcal{O}(v^1_n)}.
\end{eqnarray*}
Since on the other hand $\hat\mu(S_n)=\sum_{|\beta|=n}\hat\mu([\beta]\times[\beta])$ we get
$$
\hat\alpha_{k+1}=\lim_{n\to\infty}\frac{ \sum_{\beta\in \mathcal A^n}\hat\mu([\beta]\times[\beta])
\sum_{\alpha\in\mathcal A^k}\hat{g}_k(\alpha\beta,\alpha\beta)e^{\mathcal{O}(v^1_n)}}
{\sum_{\beta\in \mathcal A^n}\hat\mu([\beta]\times[\beta])}.
$$
For any $n\ge1$, let us define on $\Sigma_B$ the measure 
$$
\nu_n=\frac1{Z_n}\sum_{\gamma\in\mathcal A^n}\hat\mu([\gamma]\times[\gamma])\delta_\gamma,
$$
where $\delta_\gamma$ is a point mass at an arbitrarily chosen point $x_\gamma\in [\gamma]\subset\Sigma_B$
depending only on the last symbol of $\gamma$ so that $x_{\sigma\gamma}=\sigma(x_\gamma)$
and $Z_n=\sum_{|\beta|=n}\hat\mu([\beta]\times[\beta])$ is the normalising factor.
Acting on functions $f:\Sigma_B\to \mathbb{R}$ we define the transfer operator  $\mathcal{L}$  by
$$
\mathcal{L}f(x)=\sum_{a\in\mathcal{A}}g^\Delta(ax)f(ax),
$$
where $g^\Delta:\Sigma_B\to\mathbb{R}$ is given by $g^\Delta(y)=\hat{g}(y,y)$.
Then for the action of $\mathcal{L}^k$ on $\nu_n$ we get
\begin{eqnarray*}
(\mathcal{L}^k\nu_n)(f) 
&=&\int\sum_{|\beta|=k}g^\Delta_k(\beta x)f(\beta x)\,d\nu_n(x)\\
&=&\frac1{Z_n}\sum_{|\gamma|=n}\sum_{|\beta|=k}
g^\Delta_k(\beta x_\gamma)\hat\mu([\gamma]\times[\gamma])f(\beta x_\gamma)\\
&=&\frac1{Z_n}\sum_{|\alpha|=n+k}\hat\mu([\alpha]\times[\alpha])f(\alpha  x_\alpha)\\
&=&\frac{Z_{n+k}}{Z_n}e^{\mathcal{O}(v^1_n)}\nu_{n+k}(f),
\end{eqnarray*}
where we used that by conformality 
$$
\hat\mu([\gamma]\times[\gamma])=\int_{[\beta\gamma]\times[\beta\gamma]} \hat{g}_k(x,y)^{-1}\,d\hat\mu(x,y)
$$
which implies
$\hat\mu([\gamma]\times[\gamma])g^\Delta_k(\beta x_\gamma)
=\hat\mu([\beta\gamma]\times[\beta\gamma])e^{\mathcal{O}(v^1_n)}$
as $x_{\beta\gamma}=\beta x_\gamma$.
That is, we can write 
$$
\nu_n=c_{n,k}e^{\mathcal{O}(v^1_k)}\mathcal{L}^{n-k}\nu_k,
$$
where $c_{n,k}$ is a normalising constant.

Now let $\nu$ be the unique conformal measure for $e^{-P}\mathcal{L}$,
where $P$ is the pressure of $\log g^\Delta$ (on $(\Sigma_B,\sigma)$). Evidently $e^{-P}\mathcal{L}\nu=\nu$
and there is an associated positive eigenfunction $h$ so that $e^{-P}\mathcal{L}h=h$.
For simplicity's sake we assume the normalisation $\nu(1)=\nu(h)=1$. Then
$$
e^{-\ell P}\mathcal{L}^\ell(f) = h\nu(f) +\mathcal{O}(\lambda^\ell),
$$
where $\lambda<1$ as $\mathcal{L}$ is quasi compact which is a consequence of 
exponentially decaying variation of $g^\Delta$.
Thus 
$$
e^{-\ell P}(\mathcal{L}^\ell\nu_k)(f)
=e^{-\ell P}\nu_k(\mathcal{L}^\ell(f))
=\nu_k(h)\nu(f
)+\mathcal{O}(\lambda^\ell)
$$
and consequently for every $k$ and function $f$:
\begin{eqnarray*}
\lim_{n\to\infty}\nu_n(f)
&=&\lim_{n\to\infty} c_{n,k}\mathcal{L}^{n-k}\nu_k(f)e^{\mathcal{O}(v^1_k)}\\
&=&\nu_k(h)\nu(f)e^{\mathcal{O}(v^1_k)}.
\end{eqnarray*}
In particular for the constant function $f=1$ one has 
$1=\lim_{n\to\infty}\nu_n(1)=\nu_k(h)\nu(1)e^{\mathcal{O}(v^1_k)}$ which implies that
$\nu_k(h)=e^{\mathcal{O}(v^1_k)}$. If we let $k\to\infty$ we obtain that 
$\nu_k(h)\to1$ which implies that in fact $\nu_n\to \nu$ weakly.

Finally we obtain
$$
\hat\alpha_{k+1}
=\lim_{n\to\infty}\sum_{\alpha\in\mathcal A^k}\int_{\Sigma_B}g^\Delta_k(\alpha x)\,d\nu_n(x)
=\sum_{\alpha\in\mathcal A^k}\int g^\Delta_k(\alpha x)\,d\nu(x)
=\int\mathcal{L}^k1(x)\,d\nu(x)
$$
and consequently 
$$
\hat\alpha_{k+1} 
=\nu(\mathcal{L}^k1)
=e^{kP}
$$
since $\mathcal{L}^k(1)=e^{kP}h\nu(1)+R_k$ where $R_k$ is orthogonal to $h$, that is
$\nu(R_k)=0$.
This
 implies that the limiting distribution is P\'olya-Aeppli since  $P=P(\log g^\Delta)$
is negative
which follows from the fact that the pressure of $\hat g$ is zero on the system $(\Sigma_B^2, \hat\sigma)$
and that the topological entropy of $\hat\sigma$ is positive by the $\psi$-mixing property.
\end{proof}

\noindent {\bf Acknowledgements.}
SG thanks the Centre de Physique Théorique of Marseille for hospitality during part of the elaboration of this work. SG also thanks FAPESP  
(19805/2014 and 2017/07084-6) as well as CNPq Universal (439422/2018-3)  for financial support. 
SG would like to thanks Fr\'ed\'eric Paccaut for discussions concerning the  Furstenberg \& Furstenberg example.
NH was supported by Universit\'e de Toulon and the Simons Foundation (ID 526571). The research of SV was supported by the project ``Dynamics and Information Research Institute'' within the agreement between UniCredit Bank and Scuola Normale Superiore di Pisa.

\bibliographystyle{jtbnew} 
\bibliography{NSS_BIB}
\end{document}